\documentclass[final,leqno]{amsart}

\usepackage{mathrsfs}
\usepackage{rotating} 

\usepackage{amssymb,amsmath,bbm,amsfonts,latexsym,subfigure}
\usepackage{amsthm}
\usepackage{graphicx,float,epsfig,color,fancyhdr}

\def\ds{\displaystyle}

\def\a{\alpha}
\def\b{\beta}

\newcommand{\bC}{{\bf C}}

\newcommand{\bH}{{\bf H}}

\newcommand{\bS}{{\bf S}}
\newcommand{\bT}{{\bf T}}

\def\R{\mathbb R}
\def\P{\mathcal{P}}
\def\mU{\mathcal{U}}

\def\curl{\mathop{\mathbf{curl}}\nolimits}

\def\H{\mathrm H}

\def\L{\mathrm L}
\def\BK{\textbf{K}}
\def\BM{\textbf{M}}

\def\C{\mathbb C}
\def\div{\mathop{\mathrm{div}}\nolimits}

\def\ba{\boldsymbol{a}}

\def\bu{\boldsymbol{u}}

\def\bv{\boldsymbol{v}}

\def\obv{\overline{\boldsymbol{v}}}
\def\obw{\overline{\boldsymbol{w}}}
\def\bn{\boldsymbol{n}}
\def\bx{\boldsymbol{x}}
\def\bz{\widehat{\bu}}

\def\bf{\boldsymbol{f}}
\def\bg{\boldsymbol{g}}
\def\bT{\boldsymbol{T}}

\def\bB{\boldsymbol{B}}

\def\bC{\boldsymbol{C}}
\def\bId{\boldsymbol{I}}
\def\bS{\boldsymbol{S}}
\def\bH{\boldsymbol{H}}
\def\bU{\boldsymbol{U}}
\def\b0{\boldsymbol{0}}
\def\bphi{\boldsymbol{\phi}}

\def\bpsi{\boldsymbol{\psi}}
\def\hp{\widehat{p}}

\def\bw{\boldsymbol{w}}
\def\mE{\mathcal{E}}
\def\mG{\mathcal{G}}
\def\mH{\mathcal{H}}
\def\mK{\mathcal{K}}
\def\mV{\mathcal{V}}

\renewcommand\sp{\mathop{\mathrm{Sp}}\nolimits}

\newcommand\spd{\mathop{\mathrm{Sp}_\mathrm{disc}}\nolimits}
\newcommand\spe{\mathop{\mathrm{Sp}_\mathrm{ess}}\nolimits}

\def\div{\mathop{\mathrm{div}}\nolimits}

\newtheorem{lmm}{Lemma}[section]

\newtheorem{thrm}{Theorem}[section]

\newtheorem{remark}{Remark}[section]
\newtheorem{problem}{Problem}

\def\LO{\mathrm L^2(\Omega)}
\def\HdivO{{\mathrm H(\div;\Omega)}}
\begin{document}
\title[Acoustic Vibration Problem for dissipative fluids]
{Acoustic Vibration Problem for dissipative fluids}

\author{Felipe Lepe}
\address{CI$^{\mathrm{2}}$MA, Departamento de Ingenier\'{\i}a Matem\'atica,
Universidad de Concepci\'on, Casilla 160-C, Concepci\'on, Chile.}
\email{flepe@ing-mat.udec.cl}
\thanks{The first author was supported by a CONICYT fellowship (Chile).}
\author{Salim Meddahi}
\address{Departamento de Matem\'aticas, Facultad de Ciencias, 
Universidad de Oviedo, Calvo Sotelo s/n, Oviedo, Spain.}
\email{salim@uniovi.es}
\thanks{The second author was supported by
Spain's Ministry of Economy Project MTM2013-43671-P}
\author{David Mora}
\address{Departamento de Matem\'atica, Universidad del B\'io-B\'io,
Casilla 5-C, Concepci\'on, Chile and Centro de Investigaci\'on en
Ingenier\'ia Matem\'atica (CI$^{\mathrm{2}}$MA),
Universidad de Concepci\'on, Concepci\'on, Chile.}
\email{dmora@ubiobio.cl}
\thanks{The third author was partially supported by CONICYT-Chile
through FONDECYT project 1140791 (Chile) and by DIUBB through project 151408 GI/VC,
Universidad del B\'io-B\'io, (Chile)}
\author{Rodolfo Rodr\'{\i}guez}
\address{CI$^{\mathrm{2}}$MA, Departamento de Ingenier\'{\i}a Matem\'atica,  Universidad de
Concepci\'on, Casilla 160-C, Concepci\'on, Chile.}
\email{rodolfo@ing-mat.udec.cl}
\thanks{The fourth author was partially supported by BASAL project CMM,
Universidad de Chile (Chile).}



\begin{abstract}
In this paper we analyze a finite element method
for solving a quadratic eigenvalue problem
derived from the acoustic vibration problem for a
heterogeneous dissipative fluid. The problem
is shown to be equivalent to the spectral
problem for a noncompact operator and a
thorough spectral characterization is given.
The numerical discretization of the problem is
based on Raviart-Thomas finite elements.
The method is proved to be free of spurious
modes and to converge with optimal order.
Finally, we report numerical tests
which allow us to assess the performance of the method.
\end{abstract}

\maketitle

\section{Introduction}

This paper deals with the numerical approximation of an acoustic
dissipative fluid system. This kind of problem has attracted much
interest, since it is frequently encountered in engineering applications
(\cite{BGHRS,Kinsler,OS}). One typical example is to achieve optimal designs that
reduce noise and vibrations in fluid-structure systems like cars,
aircraft or dams.

Although dissipation is usually neglected in standard acoustics,
modeling this phenomenon is essential to study the effect of noise
reduction techniques. Indeed, in most real situations, damping
mechanisms that transform mechanical energy into heat do exist.
Sometimes these mechanisms are based on surface damping arising from
viscoelastic materials placed on the boundary of the propagation domain.
In these cases, the dissipative effects are typically included in the
model by means of a surface impedance in the boundary conditions (see,
for instance, \cite{BDRS,BR2,BR3}). The present paper addresses damping when it
arises in the propagation media itself due to friction and heat
conduction. A general approach to this topic can be found in the books
by Landau and Lifshitz \cite{LL}, Morse \cite{Morse}, and Pierce \cite{Pierce}, all of
which include extensive bibliographic references on the subject.

This paper focus on computing the (complex) vibration frequencies and
modes of an acoustic dissipative fluid system within a rigid cavity. One
motivation for considering this problem is that it constitutes a
stepping stone towards the more challenging goal of devising numerical
approximations for coupled systems involving fluid-structure interaction
between viscous fluids and solid structures. The natural model for the
fluid system should be based on the Stokes equations for compressible
fluids. However, since in real applications the viscosity is typically
very small, the resulting problem turns out a singular perturbation of
that for an inviscid fluid. This fact leads to a kind of dilemma, since
appropriate finite elements for the Stokes equations introduce spurious
modes in the limit case of a vanishing viscosity, whereas the finite
elements that avoid such spectral pollution fail when applied to the
Stokes equation. 

To circumvent this drawback, we resort to an alternative model based on
a curl-free displacement formulation (see \cite{BDR} for the derivation of a
similar model in the time domain from basic mechanical laws). Let us
remark that in principle the fluid displacement does not need to be
curl-free. However, since the viscosity term due to vorticity is
typically very small, except perhaps near the walls of the enclosure, it
may be neglected in the interior of the enclosure and eventually modeled
as a wall impedance on its boundary (see \cite{OS} for a similar model).

The numerical solution of the vibration problem for
an inviscid acoustic homogeneous fluid is nowadays a well known subject
(see, for instance, \cite{BGHRS}). In its turn, as is shown in Remark 2.1 of the
present paper, the vibration frequencies of a viscous homogeneous
irrotational fluid within a rigid cavity can be algebraically computed
from those of the analogous inviscid fluid and the corresponding
vibration modes coincide. However, this is not the case for a
heterogeneous fluid and this is the reason why we choose this as our
model problem. In particular, we consider the acoustic vibration problem
for a dissipative fluid system that consists of two homogeneous
viscous immiscible fluids contained in a rigid cavity.

We begin with a variational formulation of the
spectral problem relying only on the fluid displacement,
which leads to  a quadratic eigenvalue problem.
For the theoretical analysis, this  is
transformed into an equivalent double-size linear eigenvalue problem.
We introduce a convenient functional framework to
analyze it and prove that the nonlinear eigenvalue
problem is equivalent to the spectral
problem for a nonselfadjoint, noncompact bounded operator.
Thus, the essential spectrum not necessarily reduces to zero
(as is the case for compact operators).
This means that the spectrum may now contain nonzero
eigenvalues of infinite-multiplicity, nonzero accumulation points,
continuous spectrum, etc. Thus, following \cite{KL},
our first task is
to prove that the relevant eigenvalues can be isolated
from the essential spectrum, at least for
sufficiently small values of the viscosity that are realistic in practice.
Then, we propose a conforming discretization
based on Raviart-Thomas finite elements.
By appropriately adapting the abstract spectral
approximation theory for noncompact operators developed
in \cite{DNR1,DNR2}, we establish that the resulting
scheme provides a correct approximation of the spectrum
and prove error estimates for the eigenfunctions and a double
order for the eigenvalues. Moreover,
the discrete quadratic eigenvalue problem is shown
to be equivalent to a well posed generalized eigenvalue
problem which can be solved by
standard eigensolvers like {$\tt{eigs}$} from {\rm MATLAB},
which is based on Arnoldi iterations.

The paper is organized as follows: in Section~\ref{MAIN_PROBLEM},
we introduce the spectral problem and the corresponding variational
formulation, which leads to a quadratic
eigenvalue problem. We introduce an auxiliary unknown
to transform the quadratic eigenvalue problem into a linear one.
Moreover, we introduce the corresponding solution operator for the
spectral problem. In Section~\ref{SPEC_CHAR}, we provide a thorough
spectral characterization of the solution operator, based on the theory
developed in \cite{KL}. We also consider the limit problem
(i.e.,  the case when the viscosity vanishes) and the relation
between the solutions of the dissipative and non-dissipative problems.
In Section~\ref{SPEC_APP}, we introduce a finite element discretization
using Raviart-Thomas elements for both fluids and imposing the continuity
of the corresponding normal components on the interface.
We analyze the discrete spectral problem analogously as in
the continuous case and introduce the corresponding discrete
solution operator. We use the abstract theory from \cite{DNR1} to
prove the convergence. We also prove error estimates for our problem
by adapting the arguments from \cite{BDRS}. Finally,
in Section~\ref{NUMERICOS}, we report some numerical
tests which allow us to asses the performance of the proposed method.

Throughout the paper, $\Omega$ is a generic Lipschitz bounded domain of
$\R^d$ ($d=2,3$), with outer unit normal vector $\bn$. We denote by
$\mathcal{D}(\Omega)$ the space of infinitely smooth function compactly
supported in $\Omega$. For $r\geq 0$, $\left\|\cdot\right\|_{r,\Omega}$ stands indistinctly
for the norm of the Hilbertian Sobolev spaces $\H^r(\Omega)$ or $\H^r(\Omega)^d$ with
the convention $\H^0(\Omega):=\LO$. We also define the Hilbert space
$\HdivO:=\{\bv\in\LO^d:\ \div\bv\in\LO\}$, whose norm is
given by $\left\|\bv\right\|^2_{\div,\Omega}
:=\left\|\bv\right\|_{0,\Omega}^2+\left\|\div\bv\right\|^2_{0,\Omega}$,
and its subspace $\H_0(\div;\Omega):=\left\{\bv\in\HdivO:\ \bv\cdot\bn=0\,\,\textrm{on}\,\,\partial\Omega\right\}$. 
%

Finally, $C$ represents a
generic constant independent of the discretization parameters,
which may take different values at different places.
\setcounter{equation}{0}
\section{The model problem}
\label{MAIN_PROBLEM}
 We take as our model problem the case of two immiscible fluids within a rigid cavity. 
 Let $\Omega_i$ with $i=1,2$ be the polygonal (in the 2D case) or polyhedral (in the 3D case)
  Lipschitz domains
 occupied by each of the fluids.  Let $\rho_i$ be the
 corresponding densities, $\nu_i$ the fluid viscosities, and
 $c_i$ the acoustic speeds, which we consider all
 constant, $\rho_i$ and $c_i$ strictly positive and $\nu_i$ non negative. We denote by $\boldsymbol{n}_i$ the outward unit
 normal vectors corresponding to each subdomain. We define $\Omega:=(\overline{\Omega}_1\cup\overline{\Omega}_2)^{\circ}$, $\Gamma:=\partial\Omega_1\cap\partial\Omega_2$, and 
$\Gamma_i:=\partial\Omega_i\cap\partial \Omega$,  $i=1,2$. 
We assume that each domain $\Omega_i$ as well as $\Omega$ are simply connected  (see Figure~\ref{FIG:dominio}).
\begin{figure}[H]
\begin{center}
\includegraphics[height=5.7cm, width=7cm, angle=0]{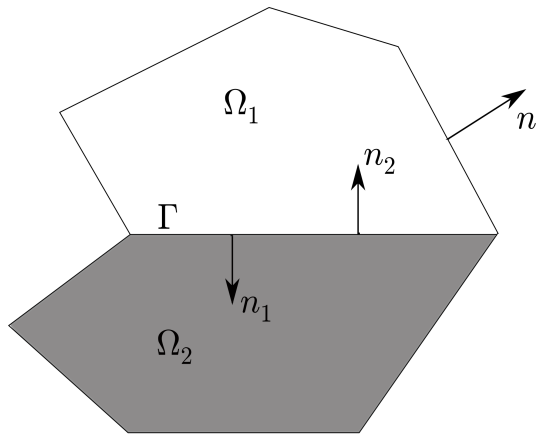}
\caption{2D sketch of the polygonal domains for the fluids.}
\label{FIG:dominio}
\end{center}
\end{figure}
We consider small displacements
of a compressible viscous fluid at rest neglecting convective terms.
The equation of motion derived from the Stokes equation reads
\begin{equation*}
\rho_i\ddot{\bU}_i=2\nu_i\Delta\dot{\bU}_i-\nabla P_i\quad\mbox{in}\ \Omega_i,
\end{equation*} 
where $\bU_i$ denotes the fluid displacement and $P_i$ the pressure
fluctuation in the domain $\Omega_i$,  $i=1,2$. The dot represents
derivation with respect to time. Moreover, since the fluid is
compressible, we consider the isentropic relation                   
 \begin{equation*}                      
P_i+\rho_ic_i^2\div\bU_i=0\qquad\mbox{in}\ \Omega_i.
\end{equation*}

Since we are considering irrotational fluids, we assume that
$\curl\bU_i=\b0$. Hence, considering the identity
$\Delta\dot{\bU_i}=\nabla(\div\dot{\bU_i})
-\curl(\curl\dot{\bU_i})$, we conclude that $\Delta\dot{\bU_i}=\nabla(\div\dot{\bU_i})$. Then, the equations of our model problem are the following:
\begin{align}
\rho_1\ddot{\boldsymbol{U}}_1 - 2\nu_1\nabla(\div\dot{\boldsymbol{U}}_1)+\nabla P_1&=\b0\quad \mbox{in}\ \Omega_1\times (0,T),\label{CONST1}\\
                        P_1 + \rho_1 c_1^2\div\boldsymbol{U}_1&=0\quad  \mbox{in}\ \Omega_1\times [0,T],\label{PRESION1}\\
\rho_2\ddot{\boldsymbol{U}}_2- 2\nu_2\nabla(\div\dot{\boldsymbol{U}}_2)+\nabla P_2&=\b0\quad \mbox{in}\ \Omega_2\times(0,T),\label{CONST2}\\
                        P_2 + \rho_2 c_2^2\div\boldsymbol{U}_2&=0\quad \mbox{in}\ \Omega_2,\times [0,T],\label{PRESION2}\\
                        \boldsymbol{U}_1\cdot\boldsymbol{n}_1+\boldsymbol{U}_2\cdot\boldsymbol{n}_2&=0\quad \mbox{on}\ \Gamma\times [0,T],\\
                        (2\nu_1\div\dot{\bU}_1+P_1)-(2\nu_2\div\dot{\bU}_2+P_2)&=0\quad \mbox{on}\ \Gamma\times (0,T),\\
                        \boldsymbol{U}_1\cdot\boldsymbol{n}_1&=0\quad \mbox{on}\hspace{0.1cm}\Gamma_1\times (0,T),\\
                        \boldsymbol{U}_2\cdot\boldsymbol{n}_2&=0\quad \mbox{on}\ \Gamma_2\times (0,T)\label{boundary2}.
       \end{align}
Let us remark that a similar argument leads exactly
to the same equations in 2D.

Multiplying  equations \eqref{CONST1} and \eqref{CONST2}
by a test function
$\bv\in\H_0(\div;\Omega)$, 
integrating by parts,
and using the boundary conditions and the transmission
conditions on $\Gamma$, we obtain
\begin{equation}\label{problem1_1}
\ds\int_{\Omega} \rho\ddot{\boldsymbol{U}}\cdot \bv+2\int_{\Omega}\nu\div\dot{\boldsymbol{U}}\div \bv-\int_{\Omega}P\div \bv=0\quad\forall \bv\in \H_0(\div,\Omega),
\end{equation}
where
\begin{equation*}
\bU := \left\{
\begin{array}{c l}
 \bU_1 & \mbox{in}\ \Omega_1,\\
 \bU_2 & \mbox{in}\ \Omega_2,
 \end{array}
\right.
\hspace{0.2cm}
P := \left\{
\begin{array}{c l}
 P_1 & \mbox{in}\ \Omega_1,\\
 P_2 & \mbox{in}\ \Omega_2,
 \end{array}
\right.
\hspace{0.2cm}
\nu := \left\{
\begin{array}{c l}
 \nu_1 & \mbox{in}\ \Omega_1,\\
 \nu_2 & \mbox{in}\ \Omega_2,
 \end{array}
\right.
\end{equation*}
\begin{equation*}
\rho := \left\{
\begin{array}{c l}
 \rho_1 & \mbox{in}\ \Omega_1,\\
 \rho_2 & \mbox{in}\ \Omega_2,
 \end{array}
\right.
\hspace{0.2cm}\mbox{and}\hspace{0.2cm}
c:= \left\{
\begin{array}{c l}
 c_1 & \mbox{in}\ \Omega_1,\\
 c_2 & \mbox{in}\ \Omega_2.
 \end{array}
\right.
\end{equation*}

Using \eqref{PRESION1} and \eqref{PRESION2} we  eliminate
$P$ in \eqref{problem1_1} and write
\begin{equation}\label{problem_rewr}
\ds\int_{\Omega} \rho\ddot{\boldsymbol{U}}\cdot \bv+2\int_{\Omega}\nu\div\dot{\boldsymbol{U}}\div \bv+\int_{\Omega}\rho c^2\div \boldsymbol{U}\div \bv=0\qquad\forall \bv\in \H_0(\div,\Omega).
\end{equation}

The \emph{vibration modes} of this problem are complex solutions
of the form $\boldsymbol{U}(\boldsymbol{x},t)=e^{\lambda t}\bu(\boldsymbol{x})$ with
$\lambda\in\mathbb{C}$. Looking for this kind of solutions leads to the following quadratic
eigenvalue problem:
\begin{problem}\label{PROBLEM1}
Find $\lambda\in\mathbb{C}$ and $\b0\ne\bu\in \H_0(\div;\Omega)$ such that
\begin{equation*}
\ds\lambda^2\int_{\Omega} \rho\bu\cdot\overline{\bv}+2\lambda\int_{\Omega}\nu\div\bu\div\overline{\bv}+\int_{\Omega}\rho c^2\div\bu\div\overline{\bv}=0\qquad\forall \bv\in \H_0(\div;\Omega).
\end{equation*}
\end{problem}

Let us remark that in absence of viscosity (i.e., $\nu=0$)
we are left with the free vibration problem of two inviscid fluids in contact (whose numerical approximation has not been analyzed either).
The eigenvalues $\lambda^2$ of such a problem are negative real numbers
(as will be proved below), so that $\lambda$ are purely imaginary,
namely, $\lambda=\pm i\omega$ with $\omega$ being the so called \textit{natural vibration
frequencies}, which correspond to periodic in time solutions
$\bU(\boldsymbol{x},t)=e^{-i\omega t}\bu({\boldsymbol{x}})$ of the time domain problem.
This is the reason why, for $\nu=0$, Problem~\ref{PROBLEM1}
is usually written as follows: Find $\omega>0$ and
$\b0\ne\bu\in\H_0(\div; \Omega)$ such that 
\begin{equation}\label{PROB_LIMITE}
\ds\int_{\Omega}\rho c^2\div\bu\div\overline{\bv}=
\omega^2\int_{\Omega}\rho\bu\cdot\overline{\bv}\qquad\forall\bv\in\H_0(\div;\Omega).
\end{equation}

In the applications, $\nu$ is typically very small.
As we will show below, in such a case there are eigenvalues
$\lambda$ of Problem~\ref{PROBLEM1} that lie close to
$\pm i\omega$ with $\omega$ being a natural vibration frequency
(i.e.,  a solution of \eqref{PROB_LIMITE}). Actually,
we will prove below that those $\lambda$ converge to
$\pm i\omega$ as $\|\nu\|_{\infty,\Omega}$ goes to zero.
On solving Problem \ref{PROBLEM1}, the aim is to
compute the eigenvalues $\lambda$ close to the smallest natural
vibration frequencies $\omega>0$, which are the most relevant in the applications.
\begin{remark}\label{ec_algebraica}
In the case of a homogeneous viscous fluid, $\rho$,
$c$ and $\nu$ are constant in the whole $\Omega$.
Then, Problem~\ref{PROBLEM1} can be written as
\begin{equation*}
\ds\lambda^2\int_{\Omega}\rho\bu\cdot\obv+\frac{2\lambda\nu+\rho c^2}{\rho c^2}\int_{\Omega}\rho c^2\div\bu\div\obv=0\qquad\forall\bv\in\H_0(\div,\Omega).
\end{equation*}
Hence, in such a case, $(\lambda,\bu)$ is an eigenpair of Problem~\ref{PROBLEM1} if and only if $-\frac{\lambda^2\rho c^2}{2\lambda\nu+\rho c^2}=\omega^2$
with $(\omega,\bu)$ being a solution to problem~\eqref{PROB_LIMITE}. Therefore,
for a homogeneous viscous fluid, $\lambda$ can be algebraically computed from the solution of \eqref{PROB_LIMITE} as follows: 
\begin{equation*}
\ds\lambda=\frac{-\nu\omega^2\pm\sqrt{\nu^2\omega^4-\rho^2c^4\omega^2}}{\rho c^2}.
\end{equation*}
\end{remark}

We denote $\mH:=\L^2(\Omega)^d$ endowed with the
weighted inner product
\begin{equation*}
\ds (\bv,\bw)_{\mH}:=\int_{\Omega}\rho\bv\cdot\obw
\end{equation*}
and $\mV:=\H_0(\div;\Omega)$ with the inner product
\begin{equation*}
\ds(\bv,\bw)_{\mV}
:=\int_{\Omega}\rho\bv\cdot\obw+\int_{\Omega}\rho c^2\div\bv\div\obw.
\end{equation*}
Notice that the inner products in $\mH$ and $\mV$ induce
norms $\|\cdot\|_{\mH}$ and $\|\cdot\|_{\mV}$ on each of these spaces equivalent to the classical
$\L^2(\Omega)^d$ and $\H(\div;\Omega)$ norms, respectively.
Therefore, when it might be convenient, we will use these classical norms.

Clearly $\lambda=0$ is an eigenvalue of Problem \ref{PROBLEM1}
with associated eigenspace
\begin{equation*}\label{eigenspace1}
\mathcal{K}=\H_0(\div^0,\Omega):=
\left\{\bv\in\H_0(\div;\Omega):\ \div\bv=0\,\,\textrm{in}\,\,\Omega\right\}.
\end{equation*}
We define:
$$\mG:=\mK^{\bot_{\mV}}=\{\bv\in \mV: (\bv,\bw)_{\mV}=0\quad\forall\bw\in\mK\}.$$
Since $\mK$ is a closed subspace of $\mV$, clearly $\mV=\mG\oplus\mK.$
Notice that $\mG$ and $\mK$ are also orthogonal in the $\mH$ inner product.
Hence,
\begin{equation*}
\mG=\{\bv\in\mV: (\bv,\bw)_{\mH}=0\hspace{0.2cm}\forall\bw\in\mK\}.
\end{equation*}

The following result brings a characterization of the space $\mG$.
\begin{lmm}\label{G}
There holds
$$\mG=\ds\frac{1}{\rho}\nabla(\H^1(\Omega))\cap\mV.$$
\end{lmm}
\begin{proof}
We will prove this result by checking the double inclusion.
Let $\bv\in \mG.$ Then, for all $\bpsi\in\mathcal{D}(\Omega)^d$,
since $\curl\bpsi\in\mK$, we have
\begin{equation*}
\ds 0=\int_{\Omega}\rho \curl\bpsi\cdot \overline{\bv}=
\int_{\Omega}\bpsi\cdot\curl\left(\rho \overline{\bv}\right).
\end{equation*}
Thus, $\curl\left(\rho \bv\right)=\b0$ in $\Omega$. Since $\Omega$ is simply connected, this implies that
there exists $\varphi\in\H^1(\Omega)$ such that
$\rho\bv=\nabla\varphi$. Hence, $\bv\in\frac{1}{\rho}\nabla(\H^1(\Omega))\cap\mV$.
Conversely, let $\bv\in\frac{1}{\rho}\nabla (\H^1(\Omega))\cap\mV$
and $\bw\in\mK$. Let $\varphi\in \H^1(\Omega)$ be such that
$\bv=\frac{1}{\rho}\nabla\varphi$. Then,
\begin{eqnarray*}
\ds(\bv,\bw)_{\mH}=\ds\int_{\Omega}\rho\left(\frac{1}{\rho}\nabla\varphi\right)
\cdot\overline{\bw}=-\ds\int_{\Omega}\varphi\div\overline{\bw}+\int_{\partial\Omega}\varphi(\overline{\bw}\cdot\boldsymbol{n})=0.
\end{eqnarray*} 
Therefore, $\bv\in\mG$. The proof is complete. 
\end{proof}

In what follows we prove additional regularity for the functions in $\mG$. {}From now and on, $s$ will denote a positive number such that the following lemma holds true.
\begin{lmm}\label{REGULARIDAD_DE_U}
There exists $s>0$ (with $s$ depending on $\rho$ and $\Omega$)
such that $\bv\in\H^s(\Omega)^d$ for all  $\bv\in\mG$ and 
\begin{equation}\label{estima}
\|\bv\|_{s,\Omega}\leq C\|\div\bv\|_{0,\Omega},
\end{equation}
where $C$ is a positive constant independent of $\bv$.
\end{lmm}
\begin{proof}
According to Lemma~\ref{G}, there exists $\varphi \in\H^1(\Omega)$
such that $\bv=\frac{1}{\rho}\nabla\varphi$. Consequently,
$\varphi \in\H^1(\Omega)/ \mathbb{C}$ is the unique solution
of the following well-posed Neumann problem:
\begin{align*}
 \ds\div\left(\frac{1}{\rho}\nabla\varphi\right)&=\div\bv\quad\mbox{in}\ \Omega,\\
   \ds\frac{1}{\rho}\frac{\partial\varphi}{\partial\bn}&=0\hspace{0.2cm}\qquad\mbox{on}\ \partial\Omega.
 \end{align*}
Hence, in the 3D case, the theorem follows from  \cite[Lemma~2.20]{PETZ2} with 
$$
s:=\frac{1}{2\pi} \min\left\{\min_{\Omega}\{\rho\},
\frac{1}{\ds\max_{\Omega}\{\rho\}}\right\}>0.
$$
For the 2D case, the theorem follows by applying \cite[Lemma~4.3]{PETZ2}. (See \cite{PETZ1} for more details.)
\end{proof}

{}From the physical point of view, the time domain problem~\eqref{problem_rewr}
is dissipative in the sense that its solution should decay as $t$ increases.
The latter happens if and only if the so called \textit{decay rate}, $\mbox{Re}(\lambda)$,
is negative. The following result shows that this is the case in our formulation.
\begin{lmm}
\label{NEGATIVERATE}
Let $(\lambda,\bu)\in\mathbb{C}\times\mV$ be a solution of
Problem~\ref{PROBLEM1}. If $\lambda\ne 0$, then $\mbox{Re}(\lambda)<0.$
\end{lmm}
\begin{proof}
 Testing Problem \ref{PROBLEM1} with $\bv=\bu$, we have that $A\lambda^2+B\lambda+C=0$, with
 $$A:=\ds\int_{\Omega}\rho|\bu|^2,\quad B:=2\int_{\Omega}\nu|\div\bu|^2,\quad\mbox{and}
 \quad C:=\int_{\Omega}\rho c^2|\div\bu|^2.$$
We observe that $A>0$,  $B\geq 0$, and $C \geq 0$. Moreover,
since $\lambda=0$ if and only if $\bu\in\mK=\H_0(\div^0,\Omega)$,
for $\lambda\ne 0$ we have that $B,C>0$, too.
Therefore, since 
$\lambda=\frac{-B\pm\sqrt{B^2-4AC}}{2A}$,  it is immediate to check that $\mbox{Re}(\lambda)<0$.
\end{proof}
\begin{remark}\label{distintodecero}
 Any eigenpair $(\lambda,\bu)$ of Problem~\ref{PROBLEM1} satisfies
\begin{equation*}
\ds\lambda^2\int_{\Omega}\rho\bu\cdot\overline{\bv}
+\int_{\Omega}(2\lambda\nu+\rho c^2)\div\bu\div\overline{\bv}=0\qquad\forall \bv\in\mV.
\end{equation*}
Since the coefficients
are constant in each subdomain, if $2\lambda\nu+\rho c^2\ne 0$ in $\Omega_i$, by testing with
$\bv\in\mathcal{D}(\Omega_i)^d$ we obtain that
$\div\bu|_{\Omega_i}\in \H^1(\Omega_i)$,  $i=1,2$.
On the other hand, if $2\lambda\nu+\rho c^2=0$ in
$\Omega_i$ ($i=1$ or $2$), then, for $\lambda\ne0$, by testing
again with $\bv\in\mathcal{D}(\Omega_i)^d$, we obtain
that $\bu=\b0$ in $\Omega_i$. Thus, in any case,
$\div\bu|_{\Omega_i}\in\H^1(\Omega_i)$,  $i=1,2.$
\end{remark}

For the theoretical analysis it is convenient
to transform Problem \ref{PROBLEM1} into a
linear eigenvalue problem. With this aim
we introduce the new variable $\bz:=\lambda\bu$, as usual
in quadratic problems, and the space $\widetilde{\mV}:=\mV\times\mH$
endowed with the corresponding product norm,
which carry us to the following:
\begin{problem}\label{PROBLEM2}
 Find $\lambda\in\mathbb{C}$ and $\b0\ne(\bu,\bz)\in \widetilde{\mV}$ such that
\begin{align}
\ds&\int_{\Omega}\rho c^2\div\bu\div \overline{\bv}=
\lambda\left(-2\int_{\Omega}\nu\div\bu\div\overline{\bv}-
\int_{\Omega}\rho\bz\cdot\overline{\bv}\right)\qquad\forall \bv\in\mV\label{linear1}, \\
\ds&\int_{\Omega}\rho\bz\cdot\overline{\widehat{\bv}}=
\lambda\int_{\Omega}\rho\bu\cdot\overline{\widehat{\bv}}\qquad\forall\widehat{\bv}\in \mH.\label{linear2}
\end{align}
\end{problem}
We observe that $\lambda=0$ is an eigenvalue of Problem~\ref{PROBLEM2}
and its associated eigenspace is $\widetilde{\mK}:=\mK\times\{0\}$.
Let $\widetilde{\mG}$ be the orthogonal complement of $\widetilde{\mK}$
in $\mV\times\mH$. Notice that $\widetilde{\mG}=\mG\times\mH$.

We introduce the sesquilinear continuous form
$a:\mV\times\mV\rightarrow\mathbb{C}$ defined by
$$a(\bu,\bv):=\ds\int_{\Omega}\rho c^2\div\bu\div\overline{\bv},$$
and the sesquilinear continuous
forms $\widetilde{a}, \widetilde{b}:\widetilde{\mV}\rightarrow\widetilde{\mV}$
 defined as follows:
$$\widetilde{a}((\bu,\bz),(\bv, \widehat{\bv})):=\int_{\Omega}\rho c^2\div\bu\div\overline{\bv}+\int_{\Omega}\rho\bz\cdot\overline{\widehat{\bv}},$$
$$\widetilde{b}((\bu,\bz),(\bv, \widehat{\bv})):=\ds-2\int_{\Omega}\nu\div\bu\div\overline{\bv}-\int_{\Omega}\rho\bz\cdot\overline{\bv}+
\int_{\Omega}\rho\bu\cdot\overline{\widehat{\bv}}.$$
In what follows we  prove that $a(\cdot,\cdot)$ and $\widetilde{a}(\cdot,\cdot)$
are elliptic in  $\mG$ and $\widetilde{\mG}$, respectively.

\begin{lmm}\label{elipticidad_continua}
The sesquilinear form $a: \mG\times\mG\rightarrow\mathbb{C}$
is $\mG$-elliptic and, consequently,
$\widetilde{a}:\widetilde{\mG}\times\widetilde{\mG}\rightarrow\mathbb{C}$
is $\widetilde{\mG}$-elliptic.
\end{lmm}
\begin{proof}
For $\bv\in\mG$ we have
\begin{equation}\label{elipticidad1}
a(\bv,\bv)=\ds\int_{\Omega}\rho c^2\div\bv\div\overline{\bv}\geq\min_{\Omega}\{\rho c^2\} \|\div\bv\|_{0,\Omega}^2.
\end{equation}
Then, the $\mG$-ellipticity of $a(\cdot,\cdot)$ follows from  Lemma~\ref{REGULARIDAD_DE_U}.  {}From this, the ellipticity of
$\widetilde{a}(\cdot,\cdot)$ in $\widetilde{\mG}=\mG\times\mH$ is immediate.
\end{proof}

Let $\bT:\widetilde{\mV}\rightarrow\widetilde{\mV}$ be the bounded
linear operator defined by $\bT(\bf,\bg):=(\bu,\bz)\in\widetilde{\mG}$,
where $(\bu,\bz)$ is the unique solution of the following problem:

$$\widetilde{a}((\bu,\bz),(\bv,\widehat{\bv}))=
\widetilde{b}((\bf,\bg),(\bv, \widehat{\bv}))
\qquad\forall(\bv,\widehat{\bv})\in\widetilde{\mG}.$$
It is easy to check that

\begin{equation}\label{z=f}
\bz=\bf\quad\text{in}\ \Omega,
\end{equation}
and
\begin{equation}\label{sol_ortogonal}
\ds \int_{\Omega}\rho c^2\div\bu\div\obv=-2\int_{\Omega}\nu\div\bf\div\overline{\bv}
-\int_{\Omega}\rho\bg\cdot\overline{\bv}\qquad\forall\bv\in\mG.
\end{equation}

As a consequence  of the above equalities, we have that $\mu=0$
is an eigenvalue of $\bT$ with associated eigenspace
$\{\b0\}\times \mG^{\bot_{\mH}}$, which is nontrivial since $\mG^{\bot_{\mH}}\supset\mK$. The following lemma shows that the nonzero eigenvalues
of $\bT$ are exactly the reciprocals of the nonzero
eigenvalues of Problem~\ref{PROBLEM2} with the same
corresponding eigenfunctions.
\begin{lmm}\label{INVERSOS}
There holds that $(\mu,(\bu,\bz))$ is an eigenpair of $\bT$ $(i.e.\,\,\bT(\bu,\bz)=\mu(\bu,\bz))$ with $\mu\ne 0$
if and only if $(\lambda,(\bu,\bz))$ is a solution of
Problem~\ref{PROBLEM2} with $\lambda=1/\mu\ne 0.$
\end{lmm}
\begin{proof}
Let $(\mu, (\bu, \bz))$ be an eigenpair of $\bT$
with $\mu\neq 0$. Hence
\begin{equation}\label{spectral1}
\widetilde{a}((\bu,\bz),(\bv,\widehat{\bv}))=\frac{1}{\mu}\widetilde{b}((\bu,\bz),(\bv,\widehat{\bv}))\qquad\forall (\bv, \widehat{\bv})\in\widetilde{\mG}.
\end{equation}
Then, according to \eqref{z=f} we have that
$\bz=\frac{1}{\mu}\bu\in\mG$. Hence, for $(\bv, \widehat{\bv})\in\widetilde{\mK}=\mK\times\{\b0\}$,
clearly $\widetilde{b}((\bu,\bz),(\bv,\widehat{\bv}))=0$
and $\widetilde{a}((\bu,\bz),(\bv,\widehat{\bv}))=0$. So, \eqref{spectral1}
holds for all $(\bv,\widehat{\bv})\in\widetilde{\mV}=\widetilde{\mG}\oplus\widetilde{\mK}$;
namely, $(\lambda,(\bu,\bz))$ with $\lambda=1/\mu$ is a solution to Problem~\ref{PROBLEM2}.

Conversely, let $(\lambda, (\bu,\bz))$ be a solution of
Problem \ref{PROBLEM2} with $\lambda\neq 0$.
Taking $\bv\in\mK$ in \eqref{linear1}, we have that
$\int_{\Omega}\rho\bz\cdot\bv=0$, which implies that $\bz\in\mG$.
On the other hand, we  observe that \eqref{linear2}
implies that $\lambda\bu=\bz\in\mG$. Hence it is easy to check
that $\bT(\bu,\bz)=\mu(\bu,\bz)$ with $\mu=1/\lambda$.
\end{proof}

\setcounter{equation}{0}
\section{Spectral Characterization}
\label{SPEC_CHAR}

The goal of this section is to characterize the spectrum of
the solution operator $\bT$. Since the inclusion
$\H_0(\div;\Omega)\hookrightarrow\L^2(\Omega)^d$ is not compact,
it is easy to check from \eqref{z=f} that $\bT$ is not compact either. However, we will show that the essential spectrum,
has to lie in a small region of the complex plane,
well separated from the isolated eigenvalues which,
according to Lemma~\ref{INVERSOS}, correspond to the
solutions of Problem~\ref{PROBLEM2}. With this aim, we will resort to the
theory described in \cite{KL} to decompose appropriately $\bT$. 
Let $\bT_1, \bT_2: \mG\rightarrow\mG$ be the operators given by
\begin{align}
\ds &\bT_1 \bf=\bu_1\in\mG:\quad
a(\bu_1,\bv)=2\int_{\Omega}\nu\div\bf\div\overline{\bv}\qquad\forall\bv\in\mG,\label{opT1}\\
\ds &\bT_2 \bg=\bu_2\in\mG:\quad
a(\bu_2,\bv)=\int_{\Omega}\rho\bg\cdot\overline{\bv}\qquad\forall\bv\in\mG.\label{opT2}
\end{align}

It is easy to check that these operators  are self-adjoint
with respect to $a(\cdot,\cdot)$. Moreover $\bT_1$ is non-negative
and $\bT_2$ is positive with respect to $a(\cdot,\cdot)$
(namely, $a(\bT_1\bv,\bv)\geq 0$ $\forall\bv\in\mG$
and $a(\bT_2\bv,\bv)>0$ $\forall\bv\in\mG$, $\bv\ne\b0$).
Moreover, we have the following result.
\begin{lmm}\label{T2COMPACTO}
The operator $\bT_2:\mG\to\mG$ is compact.
\end{lmm}
\begin{proof}

Since $a(\cdot,\cdot)$ is $\mG$-elliptic
(cf. Lemma \ref{elipticidad_continua}), applying
Lax-Milgram's Lemma, we know that problem \eqref{opT2}
is well posed and has a unique solution $\bu_2\in\mG$. Moreover, according to
Lemma~\ref{REGULARIDAD_DE_U}, we know that there exists $s>0$
such that $\bu_2\in \H^s(\Omega)^d$. On the other hand, notice that
\eqref{opT2} also holds for $\bv\in\mK$,
since in such a case $a(\bu_2,\bv)=0=\int_{\Omega}\rho\bg\cdot\overline{\bv}$
for  $\bg\in\mG$. Hence, since $\mV=\mG\oplus\mK$, we have that 
\begin{equation*}
\ds a(\bu_2,\bv)=\int_{\Omega}\rho\bg\cdot\overline{\bv}\hspace{0.4cm}\forall\bv\in\mV.
\end{equation*}
Then, by testing this equation with $\bv\in\mathcal{D}(\Omega)^d\subset\mV$,
we have that $-\nabla(\rho c^2\div\bu_2)=\rho\bg$ in $\Omega$,
so that $\rho c^2\div\bu_2\in \H^1(\Omega)$. Therefore, since $\rho$ and
$c$ are positive constants in each subdomain $\Omega_1$ and $\Omega_2$,
we have that $\div\bu_2|_{\Omega_i}\in\H^1(\Omega_i)$,  $i=1,2.$
Since the inclusions $\{v\in\L^2(\Omega): v|_{\Omega_i}\in\H^1(\Omega_i),\,i=1,2.\}\subset\L^2(\Omega)$
and $\H^s(\Omega)^{d}\subset\L^2(\Omega)^d$, are compact, we derive that $\bT_2$ is compact too.
\end{proof}


The operator $\bT$ can be written
in terms of the operators $\bT_1$ and $\bT_2$
given above as follows:
\begin{equation*}
\bT=\begin{pmatrix}
 -\bT_1&-\bT_2\\
 \bId&\b0
\end{pmatrix}.
\end{equation*}
Moreover, by defining as in \cite{KL}
the operators
$$
\bS:=\begin{pmatrix}
 \bId&\b0\\
 \b0&\bT_2^{1/2}
\end{pmatrix}
\qquad\mbox{and}\qquad
\bH:=\begin{pmatrix}
 -\bT_1&-\bT_2^{1/2}\\
 \bT_2^{1/2}&\b0
\end{pmatrix},
$$
we have that $\bS\bT=\bH\bS$. We note that the eigenvalues of $\bT$ and
$\bH$ and their algebraic multiplicities coincide. Moreover the
corresponding Jordan chains have the same length. In fact, let
$\{\bx_k\}_{k=1}^r$ be a Jordan chain associated with the eigenvalue
$\mu$ of $\bT$. Then, using the identities above,  we observe that
$$\bH\bS\bx_k=\bS\bT\bx_k=\bS(\mu\bx_k+\bx_{k-1})=\mu
\bS\bx_k+\bS\bx_{k-1}, \quad k=1,\ldots,r.$$
This shows that  $\{\bS\bx_k\}_{k=1}^r$ is a Jordan
chain of $\bH$ of the same length. Actually, the whole
spectra of $\bT$ and $\bH$ coincide as is
shown in the following result, which has been
proved in Lemma 3.2 of \cite{BDRS}. 

\begin{lmm}
\label{espectros_iguales}
 There holds
 $$\sp(\bT)=\sp(\bH).$$
 Moreover, $\spe(\bT)=\spe(\bH)$, too.
\end{lmm}

The operator $\bH$ can be written as the sum of a self-adjoint
operator $\bB$ and a compact operator $\bC$:
$$\bH=\bB+\bC\quad\mbox{with}\quad\bB:=\begin{pmatrix}
-\bT_1&\b0\\
\b0&\b0
\end{pmatrix}\quad\mbox{and}\quad
\bC:=\begin{pmatrix}
\b0&-\bT_2^{1/2}\\
\bT_2^{1/2}&\b0
\end{pmatrix}.$$
Then, applying the classical Weyl's Theorem (see \cite{RS}),
we have that $\spe(\bH)=\spe(\bB)$ and the rest of the spectrum
$\spd(\bH):=\sp(\bH)\backslash\spe(\bH)$ consists of isolated
eigenvalues with finite algebraic multiplicity.
Moreover, $\spe(\bB)=\spe(-\bT_1)\cup\{0\}$.

Our next goal is to show that the essential spectrum
of $\bT_1$ must lie in a small region of the complex plane.
Actually, we will  localize the whole spectrum
of $\bT_1$. With this aim, we analyze separately
for which values  $\mu\in\C$, the operator $(\mu\bId-\bT_1)$
is not necessarily one-to-one  and for which values it is not necessarily onto.
\begin{itemize}
\item If $(\mu\bId- \bT_1)$ is not one-to-one,
then there exists $\bf\in\mG$, $\bf\neq\b0$, such that
$\bT_1\bf=\mu\bf$, namely,
\begin{equation*}\label{ec_inyect}
\ds\mu\int_{\Omega}\rho c^2\div\bf\div\overline{\bv}=
2\int_{\Omega}\nu\div\bf\div\overline{\bv}\qquad
\forall\bv\in\mG.
\end{equation*}
Then, testing with $\bv=\bf$ and using that in each subdomain the
coefficients $\rho$ and $c$ are positive, we deduce that
\begin{equation*}
\mu=\frac{2\int_{\Omega}\nu\left|\div\bf\right|^2}
{\int_{\Omega}\rho c^2\left|\div\bf\right|^2}
\end{equation*}
(we recall that for $\b0\ne\bf\in\mG$, $\int_{\Omega}|\div\bf|^2>0$ because of Lemma \ref{REGULARIDAD_DE_U}). Hence, 
$$
\mu\in\ds\left[\frac{2\min_{\Omega}\{\nu\}}{\max_{\Omega}\{\rho c^2\}},\frac{2\max_{\Omega}\{\nu\}}{\min_{\Omega}\{\rho c^2\}}\right].
$$

\item On the other hand, $(\mu\bId- \bT_1)$ is onto if and only if 
for any $\bg\in\mG$ there exists $\bf\in\mG$
such that $\bT_1\bf=\mu\bf-\bg$,
which from \eqref{opT1}  reads
\begin{equation*}
\ds\int_{\Omega}\rho c^2\div\bg\div\obv=
\int_{\Omega}(-2\nu+\mu\rho c^2)\div\bf\div\overline{\bv}\qquad\forall\bv\in\mG.
\end{equation*}
By writing $\mu=\a+\beta i$ with $\a,\beta\in\R$, the equation above reads:
\begin{equation*}
\int_{\Omega}(-2\nu+\a\rho c^2+\rho c^2\beta i)\div\bf\div\obv=
\ds\int_{\Omega}\rho c^2\div\bg\div\overline{\bv}\qquad\forall\bv\in\mG.
\end{equation*}
We observe that for all $\beta\ne 0$ the problem above has a  solution and hence
the operator $(\mu\bId-\bT_1)$ is onto. On the other hand,
if $\beta=0$, then $\mu$ has to be real. In such a case, the operator $\bT_1$
will still be onto when $(-2\nu+\mu\rho c^2)$ has the
same sign in the whole domain $\Omega$.
This happens at least in two cases:
\begin{itemize}
\item[(i)] when $\ds\mu>\frac{2\max_{\Omega}\{\nu\}}{\min_{\Omega}\{\rho c^2\}}$,  in which case $-2\nu+\mu\rho c^2>0$,
\item[(ii)] when $\ds\mu<\frac{2\min_{\Omega}\{\nu\}}{\max_{\Omega}\{\rho c^2\}}$, in which case $-2\nu+\mu\rho c^2<0$.
\end{itemize} 
Therefore, if $(\mu\bId-\bT_1)$ is not onto, then
$\mu\in\left[\frac{2\min_{\Omega}\{\nu\}}{\max_{\Omega}\{\rho c^2\}},\frac{2\max_{\Omega}\{\nu\}}{\min_{\Omega}\{\rho c^2\}}\right]$, too.
\end{itemize}


Now we are in position to write the following spectral characterization
of the solution operator $\bT$.
\begin{thrm}
\label{CARACTERIZACION}
The spectrum of $\bT$ consists of 
$$
\spe(\bT)=\sp(\bT_1)\cup\{0\}
$$
with 
$$
\sp(\bT_1)\subset\left[\frac{2\min_{\Omega}\{\nu\}}{\max_{\Omega}\{\rho c^2\}},\frac{2\max_{\Omega}\{\nu\}}{\min_{\Omega}\{\rho c^2\}}\right]
$$
and $\spd(\bT):=\sp(\bT)\setminus\spe(\bT)$, which is a set of isolated eigenvalues of finite algebraic multiplicity.
\end{thrm}
\begin{proof}
As a consequence of  the classical Weyl's Theorem (see \cite{RS}) and 
 Lemma~\ref{espectros_iguales},
$$
\spe(\bT)=\spe(\bH)=\spe(\bB)=\spe(-\bT_1)\cup\{0\},
$$
whereas the inclusion follows from the above analyis.
\end{proof}

In what follows, we will show that for $\nu$ small enough
some of the eigenvalues of $\bT$ are well separated from
its essential spectrum.  With this end, given $\bf\in\mG$,
by testing \eqref{opT1} with $\bv=\bu_1\in\mG$ and using
\eqref{elipticidad1}, we have that 
\begin{align*}
\ds\min_{\Omega}\{\rho c^2\}\|\bu_1\|_{\div,\Omega}^2\leq a(\bu_1,\bu_1)&\leq 2\|\nu\|_{\infty,\Omega}\|\div\bf\|_{0,\Omega}\|\bu_1\|_{\div,\Omega}.
\end{align*}

Therefore $\|\bT_1\|_{\mathcal{L}(\mG\times\mG)}\rightarrow 0$
as $\|\nu\|_{\infty,\Omega}$ goes to zero. Consequently,
$\bH$ converges in norm to the operator 
$$\bH_0:=\begin{pmatrix}
 \b0&-\bT_2^{1/2}\\
 \bT_2^{1/2}&\b0
\end{pmatrix}
$$
as $\|\nu\|_{\infty,\Omega}$ goes to zero. Thus, from the
classical spectral approximation theory (see \cite{KATO}), the isolated eigenvalues
of $\bH$ converge to those of $\bH_0$. 

Since the isolated eigenvalues of $\bH$ and $\bT$ coincide
(cf. Lemma~\ref{espectros_iguales}), in order to localize
those of $\bT$, we begin by  characterizing those of
$\bH_0$. Let $\mu$ be an isolated eigenvalue of
$\bH_0$ and $(\bu,\bz)\in\mG\times\mG$ the corresponding eigenfunction.
It is easy to check that 
\begin{equation}\label{-mu2}
\bH_0\begin{pmatrix}
\bu\\
\bz
\end{pmatrix}=\mu \begin{pmatrix}
\bu\\
\bz
\end{pmatrix}\quad \Longleftrightarrow\quad \bT_2\bu=-\mu^2\bu \quad\text{and}\quad\bT_2^{1/2}\bu=\mu\bz.
\end{equation}

Since $\bT_2$ is compact, self-adjoint, and positive, its spectrum
consists of a sequence of positive eigenvalues that converge to zero
and $0$ itself. Notice that the spectrum of $\bT_2$ is related with the solution of the eigenvalue  problem \eqref{PROB_LIMITE}. In fact,
this problem has $0$ as an eigenvalue with corresponding eigenspace $\mK$. The rest of the eigenvalues $\omega^2$ are strictly
positive and the corresponding  eigenfunctions $\bu\in\mK^{\bot_{\mV}}=:\mG$, so that they are also solutions of the following problem:
Find $\omega>0$ and $\bu\in\mG$ such that
\begin{equation*}
\ds a(\bu,\bv)=\omega^2\int_{\Omega}\rho\bu\cdot\overline{\bv}\qquad\forall\bv\in\mG.
\end{equation*}

Clearly $(\omega^2,\bu)$ is an eigenpair of the above problem with $\omega>0$ if and only if $\bT_2\bu=\frac{1}{\omega^2}\bu$. Thus, by virtue of \eqref{-mu2},
we have that the eigenvalues of $\bH_0$ are given
by $\pm i/\omega$  and hence they are purely imaginary. 

Now we are in a position to establish the following result.
\begin{thrm}
For each isolated eigenvalue $\pm i/\omega$ of $\bT_2$ of algebraic multiplicity
$m$, let $r>0$ be such that the disc
$D_r:=\{ z\in\mathbb{C}:\hspace{0.2cm} |z\mp i/\omega |< r\}$
intersects $\sp(\bT_2)$ only in $\pm i/\omega$. Then, there exists
$\delta >0$ such that if $\|\nu\|_{\infty,\Omega}<\delta$,
there exist $m$ eigenvalues of $\bT$, $\mu_1,\ldots, \mu_m$, (repeated according to their respective algebraic multiplicities)
lying in the disc $D_r$. Moreover, $\mu_1,\ldots,\mu_m\rightarrow\frac{i}{\omega}$
as $\|\nu\|_{\infty,\Omega}$ goes to zero.
\end{thrm}

As claimed above, the eigenvalues of $\bT$ that are relevant
in the applications, are those which are close to $\pm i/\omega$
for the smallest positive vibration frequencies $\omega$ of \eqref{PROB_LIMITE}.
According to the above theorem, these eigenvalues are well separated
from the real  axis and, hence, from the
essential spectrum of $\bT$ (cf. Theorem~\ref{CARACTERIZACION}).

\setcounter{equation}{0}
\section{Spectral Approximation}
\label{SPEC_APP}

In this section, we propose and analyze a finite element method
to approximate the solutions of Problem~\ref{PROBLEM1}.
With this end, we introduce appropriate discrete spaces.
Let $\{\mathcal{T}_h(\Omega)\}_{h>0}$ be a family of regular partitions
of $\Omega$ such that $\mathcal{T}_h(\Omega_i):=\{T\in\mathcal{T}_h: T\subset\overline{\Omega}_i\}$
are partitions of $\Omega_i$, $i=1,2$. We introduce the lowest-order Raviart-Thomas
finite element space:
\begin{equation*}
\mV_h:=\{\bv\in\mV: \bv|_T(\boldsymbol{x})=\ba+b\boldsymbol{x},\ \ba\in\mathbb{R}^d,\  b\in\mathbb{R},\ \boldsymbol{x}\in T\}.
\end{equation*}
The discretization of Problem~\ref{PROBLEM1} reads as follows: 
\begin{problem}\label{problem3}
Find $\lambda_h\in\mathbb{C}$ and $\b0\neq\bu_h\in \mV_h$ such that
\begin{equation*}
 \ds\lambda_h^2\int_{\Omega}\rho\bu_h\cdot\overline{\bv}_h+
 2\lambda_h\int_{\Omega}\nu\div\bu_h\div\overline{\bv}_h+\int_{\Omega}\rho c^2\div\bu_h\div\overline{\bv}_h=0\qquad\forall \bv_h\in \mV_h.
\end{equation*}
\end{problem}

We proceed as we did in the continuous case and 
introduce a new discrete variable $\bz_h:=\lambda_h\bu_h$
to rewrite the problem above in the following equivalent form:
\begin{problem}\label{PROBLEM4}
 Find $\lambda_h\in\mathbb{C}$ and $\b0\ne(\bu_h, \bz_h)\in \mV_h\times \mV_h$ such that
\begin{align*}
\ds&\int_{\Omega}\rho c^2\div\bu_h\div \overline{\bv}_h=\lambda_h\left(-2\int_{\Omega}\nu\div\bu_h
\div\overline{\bv}_h-\int_{\Omega}\rho\bz_h\cdot\overline{\bv}_h\right)\qquad\forall \bv_h\in\mV_h, \\
\ds&\int_{\Omega}\rho\bz_h\cdot\overline{\widehat{\bv}}_h=
\lambda_h\int_{\Omega}\rho\bu_h\cdot\overline{\widehat{\bv}}_h\qquad
\forall\widehat{\bv}_h\in \mV_h.
\end{align*}
\end{problem}
We observe that $\lambda_h=0$ is an eigenvalue of this problem
and its associated eigenspace is $\widetilde{\mK}_h:=\mK_h\times\{0\}$
with $\mK_h:=\mK\cap\mV_h$ being the eigenspace of $\lambda_h=0$ in Problem~\ref{problem3}.
At the beginning of Section \ref{NUMERICOS}, we will show that Problem~\ref{PROBLEM4}
is  well posed, in the sense that it is equivalent to a generalized matrix eigenvalue
problem with a symmetric positive definite right-hand side matrix.

We introduce the well known Raviart-Thomas interpolation operator,
$\Pi_h: \mV\cap \H^r(\Omega)^d\rightarrow\mV_h$, $r\in(0,1]$ (see \cite{MONK}), for which there holds the approximation result
\begin{equation}\label{error_inter_RT}
\ds\|\bv-\Pi_h\bv\|_{0,\Omega}\leq Ch^r(\|\bv\|_{r,\Omega}+\|\div\bv\|_{0,\Omega})
\end{equation}
and  the commuting diagram property
\begin{equation}\label{conmutativo}
\div(\Pi_h\bv)=\P_h(\div\bv),
\end{equation}
where 
$$\P_h: \L^2(\Omega)\rightarrow\mU_h:=\{v_h\in\L^2(\Omega):\hspace{0.1cm}v_h|_T\in\P_0(T)\hspace{0.2cm}\forall T\in\mathcal{T}_h\}$$
is the standard $\L^2$-orthogonal projector. Then, for any $r\in (0,1]$ we have that
\begin{equation}\label{errorL2}
\ds\|q-\P_hq\|_{0,\Omega}\leq Ch^r\|q\|_{r,\Omega}\qquad\forall q\in\H^r(\Omega).
\end{equation}

Let $\mG_h$ be the orthogonal
complement of $\mK_h$ in $\mV_h$,
and $\widetilde{\mG}_h:=\mG_h\times\mG_h\subset\widetilde{\mV}=\mV\times\mH$
endowed with the corresponding product norm.
Note that $\mG_h\nsubseteq\mG$ and hence $\widetilde{\mG}_h\nsubseteq\widetilde{\mG}$. 

The following result provides estimates for the terms in the Helmholtz decomposition of functions in $\mG_h$.
\begin{lmm}\label{helmoltzdisc}
For any $\bv_h\in\mG_h$,
\begin{equation*}
\bv_h=\frac{1}{\rho}\nabla\xi+\chi
\end{equation*}
with $\frac{1}{\rho}\nabla\xi\in\H^{s}(\Omega)^d$ and $\chi\in\mK$ satisfying
\begin{equation*}
\left\|\frac{1}{\rho}\nabla\xi\right\|_{s,\Omega}\leq
C\|\div\bv_h\|_{0,\Omega}\qquad\mbox{and}\qquad\|\chi\|_{0,\Omega}
\leq Ch^{s}\|\div\bv_h\|_{0,\Omega}. 
 \end{equation*}
\end{lmm}
\begin{proof}
 Let $\xi\in\H^1(\Omega)/\mathbb{C}$ be a solution
of the following well-posed Neumann problem:
\begin{align*}
\ds\div\left(\frac{1}{\rho}\nabla\xi\right)&=\div\bv_h\quad\mbox{in}\ \Omega,\
\\
\ds\frac{1}{\rho}\frac{\partial\xi}{\partial\bn}&=0\hspace{0.4cm}\qquad\text{on}\ \partial\Omega.
\end{align*}

Thanks to Lax-Milgram's Lemma, there exists a unique solution
$\xi\in\H^1(\Omega)/\mathbb{C}$ of this problem.
Moreover, according to Lemmas~\ref{G} and \ref{REGULARIDAD_DE_U},
$\frac{1}{\rho}\nabla\xi\in\H^s(\Omega)^d$ and
$\|\frac{1}{\rho}\nabla\xi\|_{s,\Omega}\leq C\|\div\bv_h\|_{0,\Omega}$.
Now, let $\chi:=\bv_h-\frac{1}{\rho}\nabla\xi$.
Clearly $\div\chi=0$
and $\chi\cdot\boldsymbol{n}=0$, so
that $\chi\in\mK.$ On the other hand
\begin{equation*}
\ds\|\chi\|_{\mH}^2=\int_{\Omega}\rho\chi\cdot\left(\bv_h-\frac{1}{\rho}\nabla\xi\right).
\end{equation*}
Since $\frac{1}{\rho}\nabla\xi\in\mV\cap\H^s(\Omega)^d$,
we have that $\Pi_h(\frac{1}{\rho}\nabla\xi)$ is well defined.
Hence,
\begin{equation*}
\ds\|\chi\|_{\mH}^2=\underbrace{\int_{\Omega}\rho\chi\cdot\left[\bv_h-\Pi_h
\left(\frac{1}{\rho}\nabla\xi\right)\right]}_{(I)}
+\underbrace{\int_{\Omega}\rho\chi\cdot\left[\Pi_h\left(\frac{1}{\rho}\nabla\xi\right)
-\frac{1}{\rho}\nabla\xi\right]}_{(II)}.
\end{equation*}
For $(I)$, thanks to
\eqref{conmutativo}, $\div(\bv_h-\Pi_h(\frac{1}{\rho}\nabla\xi))
=\div\bv_h-\P_h(\div(\frac{1}{\rho}\nabla\xi))=0$. Therefore, $(\bv_h-\Pi_h(\frac{1}{\rho}\nabla\xi))\in\mK_h\subset\mK$.
Since $\frac{1}{\rho}\nabla\xi\in\mG$ and $\bv_h\in\mG_h$, we obtain
\begin{equation}
\ds(I)=\int_{\Omega}\rho\bv_h\cdot\left[\bv_h-\Pi_h\left(\frac{1}{\rho}\nabla\xi\right)\right]-\int_{\Omega}\rho\left(\frac{1}{\rho}\nabla\xi\right)\cdot\left[\bv_h-\Pi_h\left(\frac{1}{\rho}\nabla\xi\right)\right]=0.
\end{equation}

For $(II)$,
since we have already proved that $\|\frac{1}{\rho}\nabla\xi\|_{s,\Omega}\leq C\|\div\bv_h\|_{0,\Omega}$
and $\div(\frac{1}{\rho}\nabla\xi)=\div\bv_h$, from \eqref{error_inter_RT} we obtain
\begin{equation*}
\ds (II)\leq \|\chi\|_{0,\Omega}\left\|\Pi_h\left(\frac{1}{\rho}\nabla\xi\right)
-\frac{1}{\rho}\nabla\xi\right\|_{0,\Omega}
\leq Ch^s\|\chi\|_{0,\Omega}\|\div\bv_h\|_{0,\Omega},
\end{equation*}
which allows us to complete the proof.
\end{proof}

Now, we will prove that $a(\cdot,\cdot)$ and $\widetilde{a}(\cdot,\cdot)$
are elliptic in $\mG_h$ and $\widetilde{\mG}_h$, respectively.
\begin{lmm}
The sesquilinear form $a:\mG_h\times\mG_h\rightarrow\mathbb{C}$
is $\mG_h$-elliptic, with ellipticity constant not depending on $h$. Consequently,
$\widetilde{a}:\widetilde{\mG}_h\times\widetilde{\mG}_h\rightarrow\mathbb{C}$
is $\widetilde{\mG}_h$-elliptic
uniformly in $h$.
\end{lmm}
\begin{proof}
Let $\bv_h\in\mG_h$. We have that
\begin{equation*}\label{ELIPTICIDAD-DISCRETA}
\ds a(\bv_h,\bv_h)=\int_{\Omega}\rho c^2\div\bv_h\div\overline{\bv}_h\geq\min_{\Omega}\{\rho c^2\}\|\div\bv_h\|
_{0,\Omega}^2.
\end{equation*}
Now, from Lemma \ref{helmoltzdisc} we write $\bv_h=\frac{1}{\rho}\nabla\xi+\chi$ with $\frac{1}{\rho}\nabla\xi\in\H^{s}(\Omega)^d$ and $\chi\in\mK$.
Then, using  Lemma \ref{helmoltzdisc} again we obtain
\begin{equation*}
\|\bv_h\|_{0,\Omega}\leq \left\|\frac{1}{\rho}\nabla\xi\right\|_{0,\Omega}+\|\chi\|_{0,\Omega}\leq C\|\div\bv_h\|_{0,\Omega},
\end{equation*}
which together with the previous inequality allow us to conclude that $a(\cdot,\cdot)$ is $\mG_h -$ elliptic.
The $\widetilde{\mG}_h$-ellipticity
of $\widetilde{a}(\cdot,\cdot)$ is a direct consequence
of the $\mG_h$-ellipticity of $a(\cdot,\cdot)$. 
\end{proof}

Now, we are in position to introduce the discrete version of the operator $\bT$. Let $\bT_{h}:\widetilde{\mV}\rightarrow\widetilde{\mV}$
be defined by
$\bT_{h}(\bf,\bg):=(\bu_h,\bz_h)$
with $(\bu_h,\bz_h)\in\widetilde{\mG}_h$ being the solution of 
\begin{equation*}
\widetilde{a}((\bu_h,\bz_h),(\bv_h,\widehat{\bv}_h))=\widetilde{b}((\bf,\bg),(\bv_h,\widehat{\bv}_h))\qquad\forall(\bv_h,\widehat{\bv}_h)\in\widetilde{\mG}_h.
\end{equation*}

It is easy to check that $(\bu_h,\bz_h)=\bT_{h}(\bf,\bg)$ if and only if 
\begin{equation}\label{zh_proyeccion}
\bz_h=\P_{\mG_h}\bf,
\end{equation}
where $\P_{\mG_h}$ is the $\mH$-orthogonal projection
onto $\mG_h$, and $\bu_h\in\mG_h$ solves
\begin{equation}\label{sol_discreta}
\ds\int_{\Omega}\rho c^{2}\div\bu_h\div\bv_h=-2\int_{\Omega}\nu\div \bf\div\overline{\bv}_h-\int_{\Omega}\rho\bg\cdot\overline{\bv}_h\qquad\forall\bv_h\in\mathcal{G}_h.
\end{equation}

Since $\bT_h(\widetilde{\mV})\subset\widetilde{\mG}_h$, there holds $\sp(\bT_h)=\sp(\bT_h|_{\widetilde{\mG}_h})\cup\{0\}$ (cf. \cite[Lemma 4.1]{BDMRS}). Thus, we will restrict  our attention to $\bT_h|_{\widetilde{\mG}_h}$. 

As claimed above,  at the beginning of Section~\ref{NUMERICOS}, Problem
\ref{PROBLEM4} will be shown to be  equivalent to a well posed
generalized matrix eigenvalue problem. This problem has $\lambda_h=0$ as an
eigenvalue with corresponding  eigenspace $\widetilde{\mK}_h$. The rest
of the eigenvalues are related with the spectrum of
$\bT_h|_{\widetilde{\mG}_h}$ according to the following lemma.
\begin{lmm}
There holds that $(\mu_h,(\bu_h,\bz_h))$ is an eigenpair of $\bT_{h}|_{\widetilde{\mG}_h}$ with $\mu_h\ne 0$
if and only if  $(\lambda_h,(\bu_h,\bz_h))$
is a solution of Problem~\ref{PROBLEM4} with $\lambda_h=1/\mu_h$.
\end{lmm}
\begin{proof}
The proof follows essentially as that of Lemma~\ref{INVERSOS},
by using the fact that $\mV_h\times\mV_h=\widetilde{\mG}_h\oplus(\mK_h\times\mK_h).$
\end{proof}

Our next goal is to show that any isolated eigenvalue
of $\bT$ with algebraic multiplicity $m$ is approximated
by exactly $m$ eigenvalues of $\bT_{h}$ (repeated according to
their respective algebraic multiplicities) and that spurious
eigenvalues do not arise. With this end, we will adapt to our problem the
theory from \cite{BDRS}, which in turn use  arguments introduced in \cite{DNR1, DNR2} to deal with non compact operators. {}From now on,
let $\mu\in\spd(\bT)$, $\mu\ne 0$, be a fixed isolated eigenvalue of finite
algebraic multiplicity $m$. Let $\mathcal{E}$ be the invariant subspace of $\bT$
corresponding to $\mu$. Our analysis will
be based on proving the following two properties:
\begin{align*}
&\text{P1.}\quad \displaystyle\|\bT-\bT_{h}\|_h:=\sup_{\b0\ne(\bf_h,\bg_h)\in\widetilde{\mG}_h}
\frac{\|(\bT-\bT_{h})(\bf_h,\bg_h)\|_{\widetilde{\mV}}}{\Vert(\bf_h,\bg_h)\Vert_{\widetilde{\mV}}}\rightarrow 0\hspace{0.3cm}
\textrm{as} \hspace{0.15cm}h\rightarrow 0;\\
\newline
&\text{P2.}\quad \forall (\bv,\widehat{\bv})\in \mathcal{E},
\displaystyle\inf_{(\bv_h,\widehat{\bv}_h)\in\widetilde{\mG}_h}\|(\bv,\widehat{\bv})
-(\bv_h,\widehat{\bv}_h)\|_{\widetilde{\mV}}\rightarrow 0
\hspace{0.3cm} \textrm{as} \hspace{0.15cm} h\rightarrow 0.
\end{align*}

Let $(\bf_h,\bg_h)\in\widetilde{\mG}_h$ and $(\bu,\bz):=\bT(\bf_h,\bg_h)$.
{}From \eqref{sol_ortogonal},  we can write $\bu=\bu_1+\bu_2$
with $\bu_1,\bu_2\in\mG$ satisfying
\begin{equation}\label{problemu1}
\bu_1\in\mG:\hspace{0.3cm}\ds\int_{\Omega}\rho c^2\div\bu_1\div\overline{\bv}=-2\int_{\Omega}\nu\div\bf_h\div\overline{\bv}\qquad\forall\bv\in\mG,
\end{equation}
and 
\begin{equation}\label{problem_u2}
\bu_2\in\mG:\hspace{0.3cm}\ds\int_{\Omega}\rho c^2\div\bu_2\div\overline{\bv}=-\int_{\Omega}\rho\bg_h\cdot\overline{\bv}\qquad\forall\bv\in\mG.
\end{equation}

The following result states some properties
of the solutions of the problems above.
\begin{lmm}\label{u1u2}
For $(\bf_h,\bg_h)\in\widetilde{\mG}_h$, let $(\bu,\bz):=\bT(\bf_h,\bg_h)$
and consider the decomposition $\bu=\bu_1+\bu_2$ as above.
Hence, $\bu_1,\bu_2\in\H^s(\Omega)^d$, $\div\bu_1\in\mU_h$, $\div\bu_2|_{\Omega_i}\in\H^{1+s}(\Omega_i)$, $i=1,2$,  and the following estimates hold
\begin{equation}\label{regu1}
\|\bu_1\|_{s,\Omega}\leq C\|(\bf_h,\bg_h)\|_{\widetilde{\mV}},
\end{equation}
\begin{equation}\label{regu2}
\|\bu_2\|_{s,\Omega}+\|\div\bu_2\|_{1+s,\Omega_1}+\|\div\bu_2\|_{1+s,\Omega_2}
\leq C\|(\bf_h,\bg_h)\|_{\widetilde{\mV}}.
\end{equation}
\end{lmm} 
\begin{proof}
Since  $\bu_1\in\mG$, due to Lemma~\ref{REGULARIDAD_DE_U}
we have that $\bu_1\in\H^s(\Omega)^d$  and
$\|\bu_1\|_{s,\Omega}\leq C\|(\bf_h,\bg_h)\|_{\widetilde{\mV}}$.
Moreover, note that \eqref{problemu1} also holds
for $\bv\in\mK$ and hence for all $\bv\in\mV$. Then,  we write 
\begin{equation*}\label{problem_u1}
\ds\int_{\Omega}(\rho c^2\div\bu_1+2
\nu\div\bf_h)\div\overline{\bv}=0\qquad\forall\bv\in\mV.
\end{equation*}
Thus, taking test functions in $\mathcal{D}(\Omega)^d\subset\mV$
we have $\nabla(\rho c^2\div\bu_1+2\nu\div\bf_h)=0$.
Since $\rho, c, \nu$ and $\div\bf_h$ are piecewise constant,
we have that $\div\bu_1$ is piecewise constant as well; namely,  $\div\bu_1\in\mU_h$.

On the other hand, since $\bu_2\in\mG$, by applying Lemma~\ref{REGULARIDAD_DE_U} again
we have that $\bu_2\in \H^s(\Omega)^d$ and $\|\bu_2\|_{s,\Omega}\leq C\|(\bf_h,\bg_h)\|_{\widetilde{\mV}}$. To prove additional
regularity for $\div\bu_2$, we use Lemma \ref{helmoltzdisc} to write $\bg_h=\frac{1}{\rho}\nabla\xi+\chi$
with $\chi\in\mK$, $\frac{1}{\rho}\nabla\xi\in\H^s(\Omega)^d$ and $\|\frac{1}{\rho}\nabla\xi\|_{s,\Omega}\leq C\|\div\bg_h\|_{0,\Omega}$.
Moreover, since $\rho$ is constant in each subdomain $\Omega_i$, also $\nabla\xi|_{\Omega_i}\in\H^s(\Omega_i)^d$, $i=1,2$.
Then, from \eqref{problem_u2} we have that
\begin{equation*}
\ds\int_{\Omega}\rho c^2\div\bu_2\div\overline{\bv}=
-\int_{\Omega}\nabla\xi\cdot\overline{\bv}\qquad\forall\bv\in\mG.
\end{equation*}
Since the above equation trivially holds for $\bv\in\mK$ too,
it holds for all $\bv\in\mV$. Then, by testing it with $\bv\in\mathcal{D}(\Omega)^d$
we have that $\nabla(\rho c^2\div\bu_2)=-\nabla\xi\in\Omega$.
Therefore, by restricting to $\Omega_i$, $i=1,2$, we have that
$\nabla(\rho c^2\div\bu_2|_{\Omega_i})=-\nabla(\xi|_{\Omega_i})\in\H^s(\Omega_i)^d$.
Since $\rho$ and $c$ are  piecewise constant, we conclude that
$\div\bu_2|_{\Omega_i}\in\H^{1+s}(\Omega_i)$,  $i=1,2$,
and $$\|\div\bu_2\|_{1+s,\Omega_1}+\|\div\bu_2\|_{1+s,\Omega_2}\leq C\|\nabla\xi\|_{0,\Omega}\leq C\|\div\bg_h\|_{0,\Omega}.$$ 
Hence, we conclude the proof.
\end{proof}

We consider a similar decomposition in the discrete case.
For $(\bf_h,\bg_h)\in\widetilde{\mG}_h$,
let $(\bu_h,\bz_h):=\bT_h(\bf_h,\bg_h)$.
We write $\bu_h=\bu_{1h}+\bu_{2h}$ with $\bu_{1h}$ and $\bu_{2h}$
satisfying 
\begin{equation}\label{problemdisc_u1}
\bu_{1h}\in\mG_h:\hspace{0.3cm}\ds\int_{\Omega}\rho c^2\div\bu_{1h}\div\overline{\bv}_h=-2\int_{\Omega}\nu\div\bf_h\div\overline{\bv}_h\qquad\forall\bv_h\in\mG_h,
\end{equation}
and
\begin{equation}\label{problemdisc_u2}
\bu_{2h}\in\mG_h:\hspace{0.3cm}\ds\int_{\Omega}\rho c^2\div\bu_{2h}\div\overline{\bv}_h=-\int_{\Omega}\rho\bg_h\cdot\overline{\bv}_h\qquad\forall\bv_h\in\mG_h.
\end{equation}

These are the
finite element discretization of problems~\eqref{problemu1}
and \eqref{problem_u2}, respectively, and the following error estimates hold true.
\begin{lmm}\label{errores_ui}
Let $(\bf_h,\bg_h)\in\widetilde{\mG}_h$.
Let $\bu_1,\bu_2$ be the solutions of problems~\eqref{problemu1}
and \eqref{problem_u2}, respectively, and $\bu_{1h},\bu_{2h}$
those of  problems \eqref{problemdisc_u1}
and \eqref{problemdisc_u2}, respectively. Then, the following estimates hold true:
\begin{equation}\label{erroru1}
\|\bu_{1}-\bu_{1h}\|_{\div,\Omega}\leq Ch^s\|(\bf_h,\bg_h)\|_{\widetilde{\mV}},
\end{equation}
\begin{equation}\label{erroru2}
\|\bu_{2}-\bu_{2h}\|_{\div,\Omega}\leq Ch^s\|(\bf_h,\bg_h)\|_{\widetilde{\mV}}.
\end{equation}
\end{lmm}
\begin{proof}
Since $\mG_h\nsubseteq\mG$, we will resort to the second Strang
Lemma, which for
problems~\eqref{problemu1} and \eqref{problemdisc_u1}
reads as follows:
\begin{equation}\label{strangu1}
\ds\|\bu_1-\bu_{1h}\|_{\div,\Omega}\leq C\left[\inf_{\bv_{h}\in\mG_h}\|\bu_1-\bv_{h}\|_{\div,\Omega}+\sup_{\b0\neq\bv_h\in\mG_h}\frac{a(\bu_1-\bu_{1h},\bv_h)}{\|\bv_h\|_{\div,\Omega}}\right]. 
\end{equation} 
Because of Lemma~\ref{u1u2}, $\Pi_h\bu_1$ is well defined.  Since $\Pi_{h}\bu_1\in\mV_h=\mG_h\oplus\mK_h$, there exists $\widetilde{\bu}_{1h}\in\mG_h$
and $\breve{\bu}_h\in\mK_h$ such that $\Pi_{h}\bu_1=\widetilde{\bu}_{1h}+\breve{\bu}_h$.
Then, since $\bu_1-\widetilde{\bu}_{1h}$ is orthogonal to $\breve{\bu}_h$, we observe that
\begin{align}
\nonumber\|\bu_1-\widetilde{\bu}_{1h}\|_{\mV}^2 &\leq\|\bu_1-\widetilde{\bu}_{1h}\|_{\mV}^2
+\|\breve{\bu}_h\|_{\mV}^2\\
\nonumber&=\|(\widetilde{\bu}_{1h}-\bu_1)+\breve{\bu}_h\|_{\mV}^2
=\|\bu_1-\Pi_h\bu_1\|_{\mV}^2\\
\nonumber &\leq C\left(\|\bu_1-\Pi_h\bu_1\|_{0,\Omega}^2
+\|\div\bu_1-\div(\Pi_h\bu_1)\|_{0,\Omega}^2\right).
\end{align}
The first term on the right hand side above is bounded as follows:
\begin{equation*}
\|\bu_1-\Pi_h\bu_1\|_{0,\Omega}\leq Ch^s(\|\bu_1\|_{s,\Omega}+\|\div\bu_1\|_{0,\Omega})\leq Ch^s\|(\bf_h,\bg_h)\|_{\widetilde{\mV}},
\end{equation*}
where we have used \eqref{error_inter_RT}, \eqref{regu1},  and the fact that
$\|\div\bu_1\|_{0,\Omega}\leq C\|\div\bf_h\|_{0,\Omega}$, which in turn follows from \eqref{problemu1} by taking $\bv=\bf_h$. 
On the other hand,  the second term vanishes because of \eqref{conmutativo} since  $\div\bu_1\in\mU_h$
(cf. Lemma~\ref{u1u2}). Hence, $\|\bu_1-\widetilde{\bu}_{1h}\|_{\div,\Omega}\leq Ch^s \|(\bf_h,\bg_h)\|_{\widetilde{\mV}}$, which allows us to control the approximation term in \eqref{strangu1}. 

For the consistency term, it is enough to recall
 that \eqref{problemu1}
holds for all $\bv\in\mV$. Then, by using \eqref{problemdisc_u1},
it is easy to check that $a(\bu_1-\bu_{1h},\bv_h)=0$ for all
$\bv_h\in\mG_h\subset\mV$. {}From this, the Strang estimate
for  $\|\bu_1-\bu_{1h}\|_{\div,\Omega}$ reads as follows:
\begin{equation*}\label{cota1}
\ds\|\bu_1-\bu_{1h}\|_{\div,\Omega}\leq C\inf_{\bv_{h}\in\mG_h}\|\bu_1-\bv_{h}\|_{\div,\Omega}\leq Ch^s\|(\bf_h,\bg_h\|_{\widetilde{\mV}}.
\end{equation*}
Thus \eqref{erroru1} holds true.

To prove \eqref{erroru2}, we resort again to the second Strang Lemma: 
\begin{equation}\label{strangu21}
\ds\|\bu_2-\bu_{2h}\|_{\div,\Omega}\leq C\left[\inf_{\bv_{h}\in\mG_h}\|\bu_2-\bv_{h}\|_{\div,\Omega}
+\sup_{\b0\neq\bv_h\in\mG_h}\frac{a(\bu_2-\bu_{2h},\bv_h)}{\|\bv_h\|_{\div,\Omega}}\right]. 
\end{equation} 

Since $\bu_2\in\H^s(\Omega)^d$  (cf. Lemma~\ref{u1u2}),
we have that $\Pi_h\bu_2$ is well defined.
We proceed as above and write $\Pi_{h}\bu_2=\widetilde{\bu}_{2h}+\check{\bu}_h$
with $\widetilde{\bu}_{2h}\in\mG_h$ and $\check{\bu}_h\in\mK_h$ to obtain
\begin{equation}
\|\bu_2-\widetilde{\bu}_{2h}\|_{\div,\Omega}\leq
C\left[\|\bu_2-\Pi_h\bu_2\|_{0,\Omega}+\|\div\bu_2-\div(\Pi_h\bu_2)\|_{0,\Omega}\right].
\end{equation}
For the first term on the right hand side above,  \eqref{error_inter_RT} and  Lemma \ref{u1u2}
yield 
\begin{equation*}
\|\bu_2-\Pi_h\bu_2\|_{0,\Omega}\leq Ch^s(\|\bu_2\|_{s,\Omega}
+\|\div\bu_2\|_{0,\Omega})\leq Ch^s\|(\bf_h,\bg_h)\|_{\widetilde{\mV}}.
\end{equation*}
For the second term, we have from \eqref{errorL2} and from Lemma \ref{u1u2} again
\begin{align*}
\ds\|\div\bu_2-\div\Pi_h\bu_2\|_{0,\Omega}^2&=\|\div\bu_2-\P_h(\div\bu_2)\|_{0,\Omega}^2,\\
&\leq Ch(\|\div\bu_2\|_{1,\Omega_1}+\|\div\bu_2\|_{1,\Omega_2})\leq Ch\|(\bf_h,\bg_h)\|_{\widetilde{\mV}}.
\end{align*}
Hence, $\|\bu_2-\widetilde{\bu}_{2h}\|_{\div,\Omega}\leq Ch^s\|(\bf_h,\bg_h)\|_{\widetilde{\mV}}$,
which allows us to bound the approximation term in \eqref{strangu21}.

For the consistency term, given  $\bv_h\in\mG_h$ we apply Lemma~\ref{helmoltzdisc} to write 
$\bv_h=\frac{1}{\rho}\nabla\xi+\chi$ with $\frac{1}{\rho}\nabla\xi\in\H^{s}(\Omega)^d$,
$\chi\in\mK$, and $\|\chi\|_{0,\Omega}\leq Ch^s\|\div\bv_h\|_{0,\Omega}$. Then, from \eqref{problem_u2} we have
\begin{equation*}
\ds a(\bu_2,\bv_h)=\int_{\Omega}\rho c^2\div\bu_2\div\obv_h=
\int_{\Omega}\rho c^2\div\bu_2\div\left(\frac{1}{\rho}\nabla\overline{\xi}\right)
=\int_{\Omega}\bg_h\cdot\nabla\overline{\xi}.
\end{equation*}
On the other hand, from \eqref{problemdisc_u2}, 
\begin{equation*}
\ds a(\bu_{2h},\bv_h)=\int_{\Omega}\rho c^2\div\bu_{2h}\div\obv_h=\int_{\Omega}\rho\bg_h\cdot\obv_h
=\int_{\Omega}\bg_h\cdot\nabla\overline{\xi}+\int_{\Omega}\rho\bg_h\cdot\overline{\chi}.
\end{equation*}
Therefore,
\begin{equation*}
 a(\bu_2-\bu_{2h},\bv_h)=-\int_{\Omega}\rho\bg_h\cdot\overline{\chi}
\leq Ch^s\|\bg_h\|_{0,\Omega}\|\bv_h\|_{\div,\Omega}
\end{equation*}
and, hence,
\begin{equation*}\label{supremo}
\ds\sup_{\b0\ne\bv_h\in\mG_h}\frac{a(\bu_2-\bu_{2h},\bv_h)}{\|\bv_h\|_{\div,\Omega}}\leq Ch^s\|\bg_h\|_{0,\Omega},
\end{equation*}
which allows us to complete the proof.
\end{proof}

Now, we are in a position to establish the following
result.
\begin{lmm}\label{P1}
Property P1 holds true. Moreover, 
$$\|\bT-\bT_{h}\|_h\leq Ch^s.$$
\end{lmm}
\begin{proof}
For $(\bf_h,\bg_h)\in\widetilde{\mG}_h$, let
$(\bu,\bz):=\bT(\bf_h,\bg_h)$ and $(\bu_h,\bz_h):=\bT_{h}(\bf_h,\bg_h)$.
{}From  \eqref{z=f}
and \eqref{zh_proyeccion} we have that $\bz-\bz_h=\bf_h-\P_{\mG_h}\bf_h=0$. Hence, by writing $\bu=\bu_{1}+\bu_{2}$ and $\bu_h=\bu_{1h}+\bu_{2h}$ as in  Lemma~\ref{errores_ui}, we have from this lemma

\begin{equation*}
\ds\|\bT-\bT_{h}\|_h
\leq\sup_{\b0\ne(\bg_h,\bf_h)\in\widetilde{\mG}_h}
\frac{\|\bu_1-\bu_{1h}\|_{\div,\Omega}+\|\bu_2-\bu_{2h}\|_{\div,\Omega}}{\|(\bf_h,\bg_h)\|_{\widetilde{\mV}}}
\leq Ch^s.
\end{equation*}
Thus, we conclude the proof.
\end{proof}

Our next goal is to prove property P2.
With this aim, first we will prove the following additional regularity result.

\begin{lmm}\label{ESTIMACIONES_UV}
Let $(\bu,\bz)\in \mathcal{E}$. Then, $\bu,\bz\in\mG\subset\H^s(\Omega)^d$,
$\div\bu,\div\bz\in\H^{1+s}(\Omega_i)$,  $i=1,2$, and 
\begin{align}
\|\bu\|_{s,\Omega}+\|\div\bu\|_{1+s,\Omega_1}+\|\div\bu\|_{1+s,\Omega_2}\leq&
C \|(\bu,\bz)\|_{\widetilde{\mV}},\label{udivu1}\\
\|\bz\|_{s,\Omega}+\|\div\bz\|_{1+s,\Omega_1}+\|\div\bz\|_{1+s,\Omega_2}\leq&
C \|(\bu,\bz)\|_{\widetilde{\mV}}.\label{udivu2}
\end{align}
\end{lmm}
\begin{proof}
We prove the above inequalities for all the generalized
eigenfunctions of  $\bT$. Let $\{(\bu_k,\bz_k)\}_{k=1}^p$ be a Jordan chain of the operator $\bT$ associated with $\mu$.
Then, $\bT(\bu_k,\bz_k)=\mu(\bu_k,\bz_k)+(\bu_{k-1}, \bz_{k-1})$, $k=1,\ldots,p$, with $(\bu_0,\bz_0)=\b0$.
 We will use an induction argument on $k$.
Assume that $\bu_{k-1}$ and $\bz_{k-1}$ belong to $\mG$ and
satisfy \eqref{udivu1} and \eqref{udivu2}, respectively (which obviously hold for $k=1$).
First note that, because of the boundedness of $\bT$, we have
\begin{equation}\label{bounghyt}
\|(\bu_{k-1},\bz_{k-1})\|_{\widetilde{\mV}}\leq
C\|(\bu_k,\bz_k)\|_{\widetilde{\mV}}.
\end{equation}
On the other hand, by using \eqref{z=f} and \eqref{sol_ortogonal} we have that
\begin{equation}\label{identvkuk}
\mu\bz_k+\bz_{k-1}=\bu_k\hspace{0.2cm}\mbox{ in}\ \Omega
\end{equation}
and that $\mu\bu_k+\bu_{k-1}\in\mG$ satisfies
\begin{equation*}
\ds\int_{\Omega}\rho c^2\div(\mu\bu_k+\bu_{k-1})\div\obv=
-2\int_{\Omega}\nu\div\bu_k\div\obv
-\int_{\Omega}\rho\bz_k\cdot\obv\qquad
\forall\bv\in\mG.
\end{equation*}
Hence, $\bu_k,\bz_k\in\mG$.

We observe that the equation
above also holds for any $\bv\in\mK$. Then, 
\begin{equation*}
\ds\int_{\Omega}\rho c^2\div(\mu\bu_k+\bu_{k-1})\div\obv=
-2\int_{\Omega}\nu\div\bu_k\div\obv
-\int_{\Omega}\rho\bz_k\cdot\obv\qquad
\forall\bv\in\mV.
\end{equation*}
Thus, considering test functions
in $\mathcal{D}(\Omega)^d\subset\mV$ we obtain
\begin{equation}\label{nablas}
\ds\nabla((\mu\rho c^2+2\nu)\div\bu_k)=\rho\bz_k-\nabla(\rho c^2\div\bu_{k-1}).
\end{equation}
Let us assume that $\mu\rho c^2+2\nu\ne 0$ in both  $\Omega_1$ and $\Omega_2$ (we discuss the other case at the end of the proof). 
Hence, since $\rho$, $c$, and $\nu$ are constant in each $\Omega_i$, $\rho_i\bz_k-\nabla(\rho_i c_i^2\div\bu_{k-1})\in \L^2(\Omega_i)^d$,
$\div\bu_k|_{\Omega_i}\in\H^1(\Omega_i)$, and
\begin{equation*}
\|\div\bu_k\|_{1,\Omega_i}\leq C\left(\|\div\bu_{k-1}\|_{1,\Omega_i}
+\|(\bu_k,\bz_k)\|_{\widetilde{\mV}}\right),\qquad i=1,2.
\end{equation*}

Now, since $\bu_k\in\mG$, due to Lemma \ref{REGULARIDAD_DE_U} 
we have that $\bu_k\in\H^s(\Omega)^d$. Then, from \eqref{estima}
and the previous estimate we have 
\begin{equation}\label{cotauk1}
\|\bu_k\|_{s,\Omega}\leq C\left(\|\div\bu_{k-1}\|_{1,\Omega_1}
+\|\div\bu_{k-1}\|_{1,\Omega_2}+\|(\bu_k,\bz_k)\|_{\widetilde{\mV}}\right).
\end{equation}
On the other hand, from \eqref{identvkuk} we obtain
\begin{equation}\label{cotavk1}
\ds\|\bz_k\|_{s,\Omega}\leq \frac{1}{\mu} \left(\|\bz_{k-1}\|_{s,\Omega}+\|\bu_k\|_{s,\Omega}\right)
\end{equation}
and, from \eqref{nablas},
\begin{equation}\label{cotadivuk1}
\|\div\bu_k\|_{1+s,\Omega_i}\leq C(\|\div\bu_{k-1}\|_{1+s,\Omega_i}+\|\bz_k\|_{s,\Omega}),\qquad i=1,2.
\end{equation}
Finally, from \eqref{identvkuk} again,
\begin{equation}\label{cotadiv_vk1}
\ds\|\div\bz_k\|_{1+s,\Omega_i}\leq\frac{1}{\mu}(\|\div\bu_k\|_{1+s,\Omega_i}
+\|\div\bz_{k-1}\|_{1+s,\Omega_i}),\qquad i=1,2.
\end{equation}
Hence, from  inequalities \eqref{cotauk1}--\eqref{cotadiv_vk1}, the inductive assumption, 
and \eqref{bounghyt}, we derive \eqref{udivu1} and \eqref{udivu2} provided $\mu\rho c^2+2\nu\ne 0$ in both $\Omega_1$ and $\Omega_2$.

In case that $\mu\rho c^2+2\nu$ vanishes in $\Omega_i$, $i=1$ or $2$,
arguing as in Remark~\ref{distintodecero} we obtain that
$\bu_1|_{\Omega_i}=\bz_1|_{\Omega_i}=\b0$ and, once again, an
induction argument allow us to conclude that $\bu_k, \bz_k=\b0$ in $\Omega_i$,
$k=1,\ldots,p$. The proof is complete.
\end{proof}

Now, we are in position to establish
property P2.

\begin{lmm}\label{REG_AUT}
Property P2 holds true. Moreover, for any  $(\bu,\bz)\in\mE$, there exists $\widetilde{\bu}_h$,
$\widetilde{\bz}_h\in\mG_h$ such that
\begin{equation*}
\|\bu-\widetilde{\bu}_h\|_{\div,\Omega}+\|\bz-\widetilde{\bz}_h\|_{\div,\Omega}\leq
Ch^s\|(\bu,\bz)\|_{\widetilde{\mV}}.
\end{equation*}
\end{lmm}
\begin{proof}
Let $(\bu,\bz)\in\mE$. According to Lemma~\ref{ESTIMACIONES_UV}
$\bu,\bz\in\H^s(\Omega)^d$ and $\div\bu,\div\bz\in\H^{1+s}(\Omega_i)$, $i=1,2$.
Let $\Pi_h\bu\in\mV_h$ be the Raviart-Thomas interpolant of $\bu$. Since
$\mV_h=\mG_h\oplus\mK_h$, we decompose
$\Pi_h\bu=\widetilde{\bu}_h+\breve{\bu}_h$ with $\widetilde{\bu}_h\in\mG_h$
and $\breve{\bu}_h\in\mK_h$. The same arguments from the proof of
Lemma~\ref{errores_ui} that lead to \eqref{erroru2} apply in this case
and combined with Lemma~\ref{ESTIMACIONES_UV} allow us to prove that
$\|\bu-\widetilde{\bu}_h\|_{\div,\Omega} \leq
Ch^s\|(\bu,\bz)\|_{\widetilde{\mV}}$. A similar procedure can be used to
define $\widetilde{\bz}_h$ and to prove that $\|\bz-\widetilde{\bz}_h\|_{\div,\Omega}\leq Ch^s\|(\bu,\bz)\|_{\widetilde{\mV}}$.
\end{proof}
  
We also have the following auxiliary result when the source terms are in $\mE$.
\begin{lmm}\label{CEONVERGENCE_EV}
For $(\bf,\bg)\in \mE$,  let $(\bu,\bz):=\bT(\bf,\bg)$
and $(\bu_h,\bz_h):=\bT_{h}(\bf,\bg)$. Then,
\begin{equation*}
\|\bu-\bu_h\|_{\div,\Omega}+\|\bz-\bz_h\|_{0,\Omega}\leq Ch^s \|(\bf,\bg)\|_{\widetilde{\mV}}.
\end{equation*}
\end{lmm}
\begin{proof}
Since  $\mG_h\nsubseteq\mG$, we resort once more to the second Strang
Lemma, which applied now to \eqref{sol_ortogonal} and \eqref{sol_discreta} leads to
\begin{equation*}\label{stranginE}
\ds\|\bu-\bu_{h}\|_{\div,\Omega}\leq C\left[\inf_{\bv_{h}\in\mG_h}\|\bu-\bv_{h}\|_{\div,\Omega}
+\sup_{\b0\neq\bv_h\in\mG_h}\frac{a(\bu-\bu_{h},\bv_h)}{\|\bv_h\|_{\div,\Omega}}\right]. 
\end{equation*}
{}From Lemma~\ref{REG_AUT} we know that there exists $\widetilde{\bu}_h\in\mG_h$ such that
\begin{equation*}\label{u-uh}
\|\bu-\widetilde{\bu}_h\|_{\div,\Omega}\leq Ch^s\|(\bu,\bz)\|_{\widetilde{\mV}}\leq Ch^s\|(\bf,\bg)\|_{\widetilde{\mV}}.
\end{equation*}
Moreover, the consistency term above vanishes.
In fact, consider $\bv_h\in\mG_h$ and the decomposition
$\bv_h=\frac{1}{\rho}\nabla\xi+\chi$ as in Lemma \ref{helmoltzdisc}.
Using the same arguments as in the proof of Lemma~\ref{errores_ui},
we prove that 
\begin{equation*}
\ds a(\bu-\bu_h,\bv_h)=\int_{\Omega}\rho\bg\cdot\overline{\chi}=0,
\end{equation*}
where the last equality holds because $\bg\in\mG$ and $\chi\in\mK$.
 
On the other hand, we know from \eqref{z=f} and \eqref{zh_proyeccion}
that $\bz=\bf$  and   $\bz_h=\P_{\mG_h}\bf$, respectively.
Then, since $\P_{\mG_h}$ is the $\mH$-orthogonal projection onto $\mG_h$,
we have that $\|\bz-\bz_h\|_{\mH}\leq\|\bz-\widetilde{\bz}_h\|_{\mH}$,
with $\widetilde{\bz}_h\in\mG_h$ as in Lemma~\ref{REG_AUT}.
Hence, we obtain
$$\|\bz-\bz_h\|_{0,\Omega}\leq Ch^s\|(\bu,\bz)\|_{\widetilde{\mV}}\le
Ch^s\|(\bf,\bg)\|_{\widetilde{\mV}}.$$
The proof is complete.
\end{proof}

The above lemmas are the ingredients to prove spectral convergence and to obtain error estimates. Our first result is the following theorem which has been proved in \cite{DNR1} as a consequence of property P1 (cf. Lemma~\ref{P1}) and which shows that
the proposed method is free of spurious modes.
\begin{thrm}
Let $K\subset\mathbb{C}$ be a compact set such that $K\cap\sp(\bT)=\emptyset$. Then,
there exists $h_0>0$ such that, for all $h\leq h_0$,  $K\cap\sp(\bT_{h})=\emptyset.$ 
\end{thrm}

Let $D\subset\mathbb{C}$ be a closed disk centered at $\mu$, such
that $D\cap\sp(\bT)=\{\mu\}$. Let $\mu_{1h},\ldots,\mu_{m(h)h}$
be the eigenvalues of $\bT_h$ contained in $D$
(repeated according to their algebraic multiplicities).
Under assumptions P1 and P2, it is proved in \cite{DNR1}
that $m(h)=m$ for $h$ small enough and that $\lim_{h\rightarrow 0}\mu_{kh}=\mu$
for $k=1,\ldots, m$. 

On the other hand the arguments used in Section 5 of \cite{BDRS}
can be readily adapted to our problem, to obtain error estimates. We recall the definition of the gap between two closed subspaces
$\mathcal{W}$ and $\mathcal{Y}$ of $\widetilde{\mV}$:
\begin{equation*}
 \widehat{\delta}(\mathcal{W},\mathcal{Y}):=\max\{\delta(\mathcal{W},\mathcal{Y}),
 \delta(\mathcal{Y},\mathcal{W})\},
\end{equation*}
with
\begin{equation*}
 \ds\delta(\mathcal{W},\mathcal{Y}):=\sup_{\stackrel{\scriptstyle (\bphi,\bpsi)\in\mathcal{W}}{\|(\bphi,\bpsi)\|_{\widetilde{\mV}}=1}}\left[\inf_{(\widehat{\bphi},\widehat{\bpsi})\in\mathcal{Y}}\|(\bphi,\bpsi)-(\widehat{\bphi},\widehat{\bpsi})\|_{\widetilde{\mV}}\right].
\end{equation*}

Let $\mathcal{E}_h$ be the invariant subspace of $\bT_h$ relative
to the eigenvalues $\mu_{1h},\ldots,\mu_{mh}$ converging to $\mu$.
{}From Lemmas \ref{P1}--\ref{CEONVERGENCE_EV}, we derive the following results for which we do not include proofs
to avoid repeating step by step those of \cite[Section~5]{BDRS}.
\begin{thrm}\label{gap-eigenspaces}
There exist constants $h_0>0$ and $C>0$ such that, for all $h\leq h_0$,   
\begin{equation*}
\widehat{\delta}\left(\mathcal{E}_h,\mathcal{E}\right)\leq Ch^s.
\end{equation*}
\end{thrm}

\begin{thrm}\label{Eigen_Aprox}
There exist constants $h_0>0$ and $C>0$ such that, for all $h\leq h_0$,
\begin{equation*}
\ds\left|\mu-\frac{1}{m}\sum_{k=1}^m\mu_{kh}\right|\leq Ch^{2s},
\end{equation*}
\begin{equation*}
\ds\left|\frac{1}{\mu}-\frac{1}{m}\sum_{k=1}^m\frac{1}{\mu_{kh}}\right|\leq Ch^{2s},
\end{equation*}
\begin{equation*}
\ds\max_{k=1,\ldots, m}|\mu-\mu_{kh}|\leq Ch^{2s/p},
\end{equation*}
where $p$ is the ascent of the eigenvalue $\mu$ of $\bT$.
\end{thrm}

\setcounter{equation}{0}
\section{Numerical Results}
\label{NUMERICOS}
We report in this section the results of some numerical tests, in order to assess the
performance of the proposed method. With this end,
first we introduce a convenient matrix form of
the discrete problem which allows us to use standard eigensolvers. As a by-product, this matrix form also allows us to prove that Problems \ref{problem3} and \ref{PROBLEM4} are well posed.

Let $\{\phi_j\}_{j=1}^N$
be a nodal basis of $\mV_h$. We define the matrices
$\BK_1:=(\BK^{(1)}_{ij})$, $\BK_2:=(\BK^{(2)}_{ij})$ and $\BM:=(\BM_{ij})$ as follows:
\begin{equation*}
\ds\BK^{(1)}_{ij}:=2\int_{\Omega}\nu\div\phi_i\div\phi_j,\ \ 
\BK^{(2)}_{ij}:=\int_{\Omega}\rho c^2\div\phi_i\div\phi_j,\ \ 
\hbox{and}\ \ \BM_{ij}:=\int_{\Omega}\rho \phi_i\cdot\phi_j.
\end{equation*}

The matrix form of Problem \ref{problem3} reads 
\begin{equation}\label{matrixform1}
(\lambda^2_h\BM+\lambda_h\BK_1+\BK_2)\vec{\bu}_h=\b0,
\end{equation}
where we denote by $\vec{\bu}_h$ the vector
of components of $\bu_h$ in the nodal basis of $\mV_h$.

Analogously, the matrix form of Problem~\ref{PROBLEM4} reads
$$
\begin{pmatrix}
\BK_2&\b0\\
\b0&\BM
\end{pmatrix}
\begin{pmatrix}
\vec{\bu}_h\\
\vec{\bz}_h
\end{pmatrix}
=\lambda_h\begin{pmatrix}
-\BK_1&-\BM\\
\BM&\b0
\end{pmatrix}
\begin{pmatrix}
\vec{\bu}_h\\
\vec{\bz}_h
\end{pmatrix},
$$
with $\vec{\bz}_h$ being the vector of components of $\bz_h$. However, this problems is not suitable to be solved with standard eigensolvers, since neither the right-hand side nor the left-hand side matrix are Hermitian and positive definite. 

Alternatively, for $\lambda_h\neq 0$, let $\mu_h:=\frac{1}{\lambda_h}$. Then,
problem \eqref{matrixform1} is equivalent to
\begin{equation*}\label{matrixform2}
(\BM+2\mu_h\BK_1+\mu_h^2\BK_2)\vec{\bu}_h=\b0.
\end{equation*}
Introducing $\vec{\bw}_h:=\mu_h\vec{\bu}_h$, the problem above is equivalent to

\begin{equation*}
\begin{pmatrix}
\BM&\b0\\
\b0&\BM
\end{pmatrix}
\begin{pmatrix}
\vec{\bu}_h\\
\vec{\bw}_h
\end{pmatrix}
=\mu_h
\begin{pmatrix}
-\BK_1&-\BK_2\\
\BM&\b0
\end{pmatrix}
\begin{pmatrix}
\vec{\bu}_h\\
\vec{\bw}_h
\end{pmatrix},
\end{equation*}
which in turn is equivalent to
\begin{equation*}\label{matrixproblem}
\begin{pmatrix}
-\BK_1&-\BK_2\\
\BM&\b0
\end{pmatrix}
\begin{pmatrix}
\vec{\bu}_h\\
\vec{\bw}_h
\end{pmatrix}
=\lambda_h
\begin{pmatrix}
\BM&\b0\\
\b0&\BM
\end{pmatrix}
\begin{pmatrix}
\vec{\bu}_h\\
\vec{\bw}_h
\end{pmatrix}.
\end{equation*}
Thus, the last problem is equivalent to Problem \ref{problem3}
except for $\lambda_h=0$ and the matrix in its right-hand side
is Hermitian and positive definite. Hence, it is well posed and
can be safely solved by standard eigensolvers.

We implemented the proposed method in a MATLAB code.
We applied it to a 2D rectangular rigid cavity filled
with two fluids with different physical parameters
as shown in Figure~\ref{FIG:cavity}.
The domain occupied by the fluids are $\Omega_1:=(0,A)\times (0,H)$
and $\Omega_2:=(0,A)\times (H,B)$. For such a simple geometry,
it is possible to calculate an analytical solution
which will be used to validate our method.

\begin{figure}[H]
\begin{center}
\includegraphics[height=5.7cm, width=4cm, angle=0]{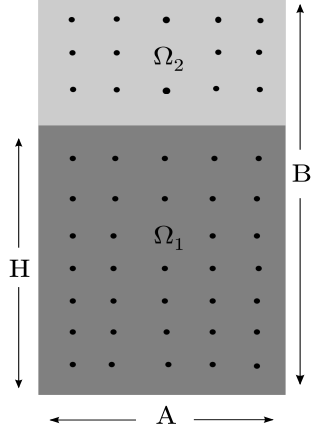}
\caption{Two fluids in a rectangular rigid cavity.}
\label{FIG:cavity}
\end{center}
\end{figure} 

Let $\bu\in\H_0(\div,\Omega)$ be a solution of Problem~\ref{PROBLEM1}.
Testing it with $\bv\in\mathcal{D}(\Omega)^2$ we have
$\nabla((2\lambda\nu+\rho c^2)\div\bu)=-\lambda^2\rho\bu\in\L^2(\Omega)^2$.
Then, $\hp:=-(2\nu\lambda+\rho c^2)\div\bu\in\H^1(\Omega)$.
Hence, $\hp_1|_{\Gamma}=\hp_2|_{\Gamma}$. Moreover,
$\ds\bu=-\frac{1}{\lambda^2\rho}\nabla\hp$, which implies that
$\ds\frac{1}{\rho_1}\frac{\partial\hp_1}{\partial\nu}
=\frac{1}{\rho_2}\frac{\partial\hp_2}{\partial\nu}$
on $\Gamma$. 
Then, we write problem~\eqref{CONST1}--\eqref{boundary2},
in terms of $\hp_i$ as follows:
\begin{align*}
\Delta\hp_i&=\frac{\lambda^2\rho_i}{\rho_i c_i^2+2\nu_i\lambda}\hp_i\hspace{0.6cm}\mbox{in}\ \Omega_i,\hspace{0.2cm}i=1,2,\\
\frac{\partial\hp_i}{\partial\bn_i}&=0\hspace{2.5cm}\mbox{on}\ \Gamma_i,\hspace{0.2cm}i=1,2,\\
\hp_1&=\hp_2\hspace{2.35cm}\mbox{on}\ \Gamma,\\
\frac{1}{\rho_1}\frac{\partial\hp_1}{\partial\bn}&=\frac{1}{\rho_2}\frac{\partial\hp_2}{\partial\bn}\hspace{1.65cm}\mbox{on}\ \Gamma.
\end{align*}

We proceed by separation of variables. Assuming that $\hp_i(x,y)=X_i(x)Y_i(y)$, we are left with the following problem:
\begin{align}
\ds \frac{X_i''(x)}{X_i(x)}+\frac{Y_i''(y)}{Y_i(y)}&= \frac{\lambda^2\rho_i}{\rho_i c_i^2+2\nu_i\lambda}\quad\mbox{in}\ \Omega_i,\label{Xi}\\
X_i'(0)=X_i'(A)&=0,\hspace{1.1cm} \qquad i=1,2,\label{contxi}\\
Y_1'(0)=Y_2'(B)&=0,\label{contyi}\\
\frac{1}{\rho_1}X_1(x)Y_1'(H)&=\frac{1}{\rho_2}X_2(x)Y_2'(H),\hspace{0.2cm}0<x<A,\label{contrhoi}\\
X_1(x)Y_1(H)&=X_2(x)Y_2(H),\quad 0<x<A\label{contidevi}.
\end{align}
{}From \eqref{Xi} we have that $X_i(x)''/X_i(x)$ and $Y_i(y)''/Y_i(y)$ are constant.
Moreover, from \eqref{contrhoi} and \eqref{contidevi},
it is easy to check that $Y_i(H)$ and $Y'_i(H)$ cannot vanish simultaneously
and  $X_1(x)=X_2(x)$ (actually, it is  derived that $X_1(x)=CX_2(x)$,
but the constant $C$ can be chosen equal to one without loss of generality).

{}From the fact that $X_i(x)''/X_i(x)$ is constant and \eqref{contxi}, we have that
\begin{equation*}
X_1(x)=X_2(x)=\cos\left(\frac{m\pi x}{A}\right),\qquad m=0,1,2,\ldots.
\end{equation*}
On the other hand, from the fact that $Y_i(y)''/Y_i(y)$ is also constant and 
\eqref{contyi} we derive 
\begin{equation}\label{Y1Y2}
Y_1(y)=C_1\cosh(r_m^{(1)}(\lambda)y)\qquad\text{and}\qquad Y_2(y)=C_2\cosh(r_m^{(2)}(\lambda)(y-B)),
\end{equation}
where $C_1$ and $C_2$ are constants and 
\begin{equation*}
r_m^{(i)}:=\sqrt{\frac{\lambda^2\rho_i}{\rho_i c_i^2+2\nu_i\lambda}+\frac{m^2\pi^2}{A^2}}, \qquad m=0,1,2,\ldots, \qquad i=1,2.
\end{equation*}

 Since $Y_i(H)$ and $Y_i'(H)$ cannot vanish simultaneously, \eqref{contrhoi}
 and  \eqref{contidevi} lead to 
\begin{equation*}\label{INT}
\frac{1}{\rho_1}Y_1'(H)=\frac{1}{\rho_{2}}Y_2'(H)\qquad\mbox{and}\qquad Y_1(H)=Y_2(H),
\end{equation*}
respectively. Thus, substituting \eqref{Y1Y2} into these equation yields
the following linear system  for the coefficients $C_1$ and $C_2$:
\begin{align*}
\ds C_1\cosh(r_{m}^{(1)}(\lambda)H)&=C_2\cosh(r_{m}^{(2)}(\lambda)(H-B)),\\
\ds\frac{C_1r_{m}^{(1)}(\lambda)}{\rho_1}\sinh(r_{m}^{(1)}(\lambda)H)&=\frac{C_2r_{m}^{(2)}(\lambda)}{\rho_2}\sinh(r_{m}^{(2)}(\lambda)(H-B)).
\end{align*}
For this system to have non trivial solutions, its determinant must vanish,
which yields the following non linear equation in $\lambda$ for $m=0,1,2,\ldots$,
whose roots are the eigenvalues of Problem~\ref{PROBLEM1}:
\begin{multline*}
f_m(\lambda):=\frac{r_{m}^{(1)}(\lambda)}{\rho_1}\sinh(r_{m}^{(1)}(\lambda)H)\cosh(r_{m}^{(2)}(\lambda)(H-B))\\
-\frac{r_{m}^{(2)}(\lambda)}{\rho_2}\sinh(r_{m}^{(2)}(\lambda)(H-B))\cosh(r_{m}^{(1)}(\lambda)H)=0.
\end{multline*}

We have computed some roots of the above equation and
used these roots as exact eigenvalues to compare those
obtained with the method proposed in this paper. 
For the geometrical parameters, we have taken $A=1\, \mbox{m}$,
$B=2\,\mbox{m}$ and $H=1.25\,\mbox{m}$. 

We have used physical parameters of water and air for the density and acoustic speed of the fluids
in $\Omega_1$ and $\Omega_2$, respectively: $c_1=1430 \hspace{0.1cm}\mbox{m/s}$,
$\rho_1=1000\hspace{0.1cm}\mbox{ kg/m}^3$, $c_2=340\hspace{0.1cm}\mbox{m/s}$
and $\rho_2=1\hspace{0.1cm}\mbox{kg/m}^3$. We have used uniform meshes
as those shown in Figure~\ref{FIG:MESH}.
The refinement parameter $N$ refers to the number of elements per width of the rectangle.

\begin{figure}[H]
\begin{center}
\begin{minipage}{6.0cm}
\centerline{\includegraphics[height=5.5cm, angle=-90]{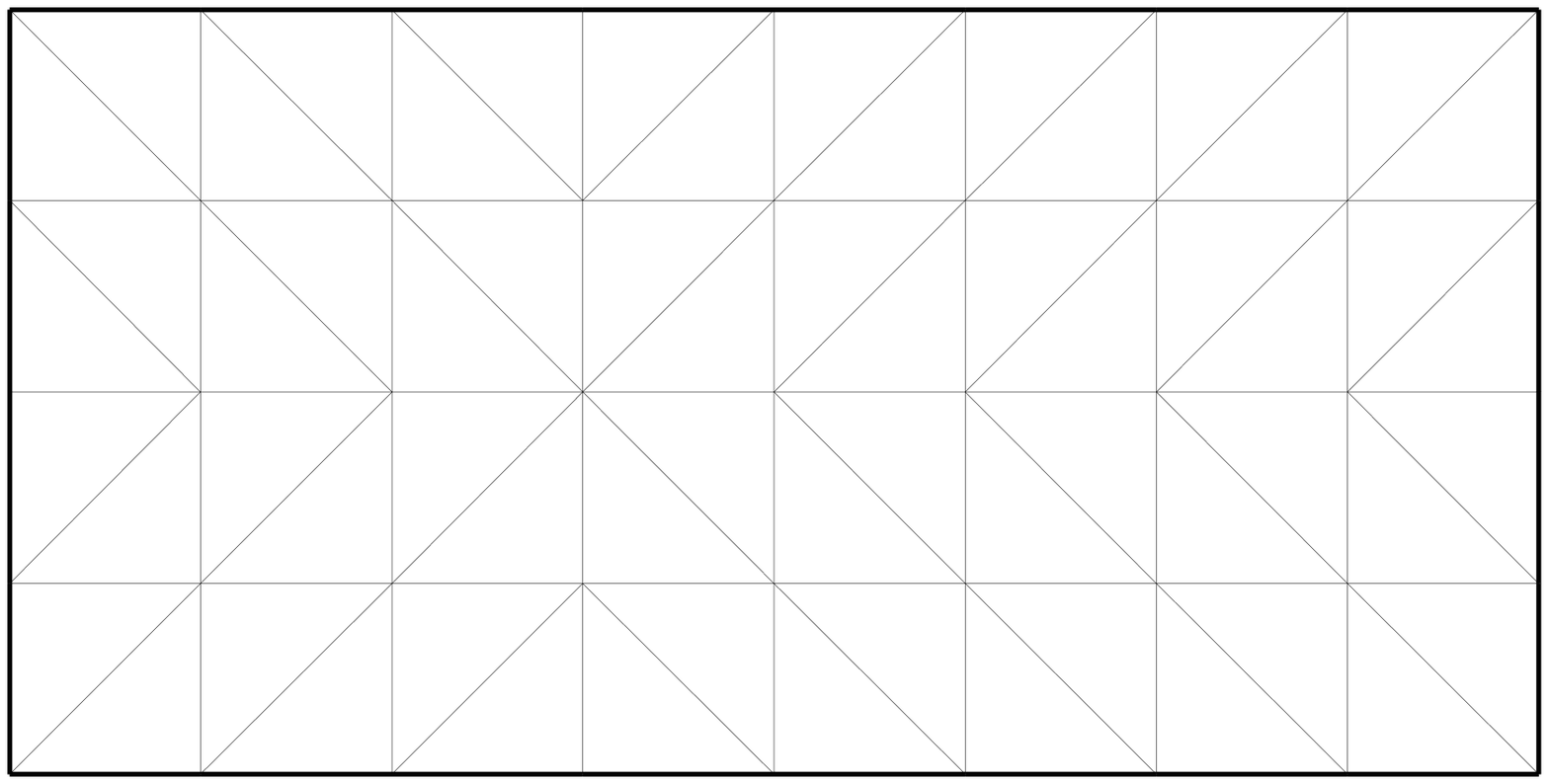}}
\end{minipage}
\begin{minipage}{6.0cm}
\centerline{\includegraphics[height=6.0cm, angle=-90]{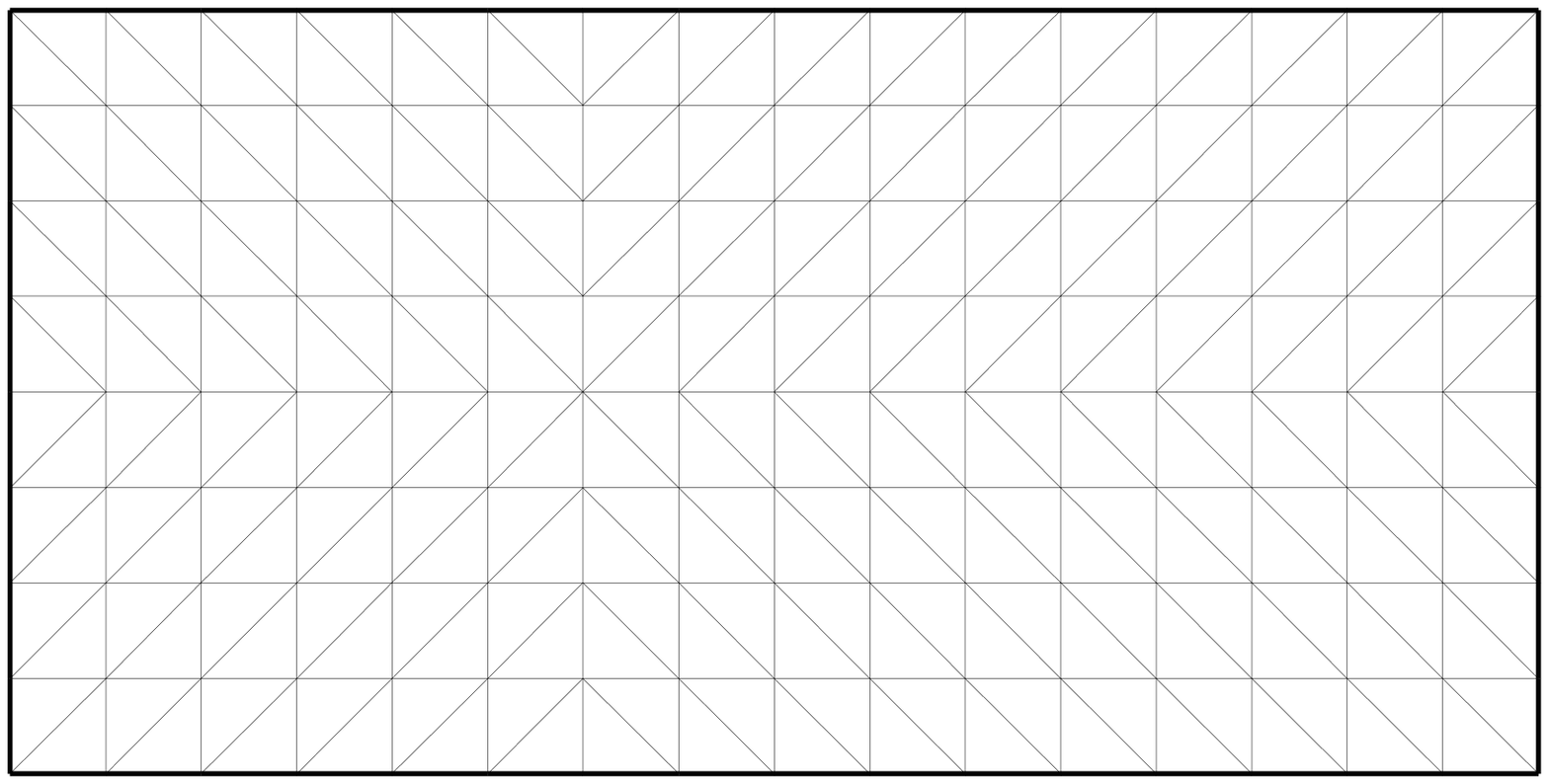}}
\end{minipage}
\end{center}
\caption{Meshes for $N=4$ (left) and $N=8$ (right).}
\label{FIG:MESH}
\end{figure}

In presence of dissipation ($\nu\ne 0$), the eigenvalues $\lambda$
 are complex numbers $\lambda=\eta+i\omega$,
with $\eta<0$ being the decay rate and $\omega$
the vibration frequency. In absence of dissipation
($\nu=0$), the eigenvalues $\lambda$ are purely imaginary ($\eta=0$).
The same holds for the computed eigenvalues $\lambda_h$.

In our first test, we neglected the viscosity damping effects
by taking $\nu_1=\nu_2=0$. In this case, the eigenvalues $\lambda$
are actually purely imaginary as can be seen from Figures~\ref{FIG:CONTOUR1}
and \ref{FIG:CONTOUR2}, which shows contour plots of $\log(|f_m(\lambda)|)$
for the smallest values of $m$ ($0\leq m\leq 3$).
Accurate values of the zeros of $f_m(\lambda)$ have been obtained
with the MATLAB command $\tt{fminsearch}$ applied to $|f_m(\lambda)|$.

\begin{figure}[H]
\begin{center}
\begin{minipage}{5.5cm}
\centerline{\includegraphics[height=5.5cm, angle=0]{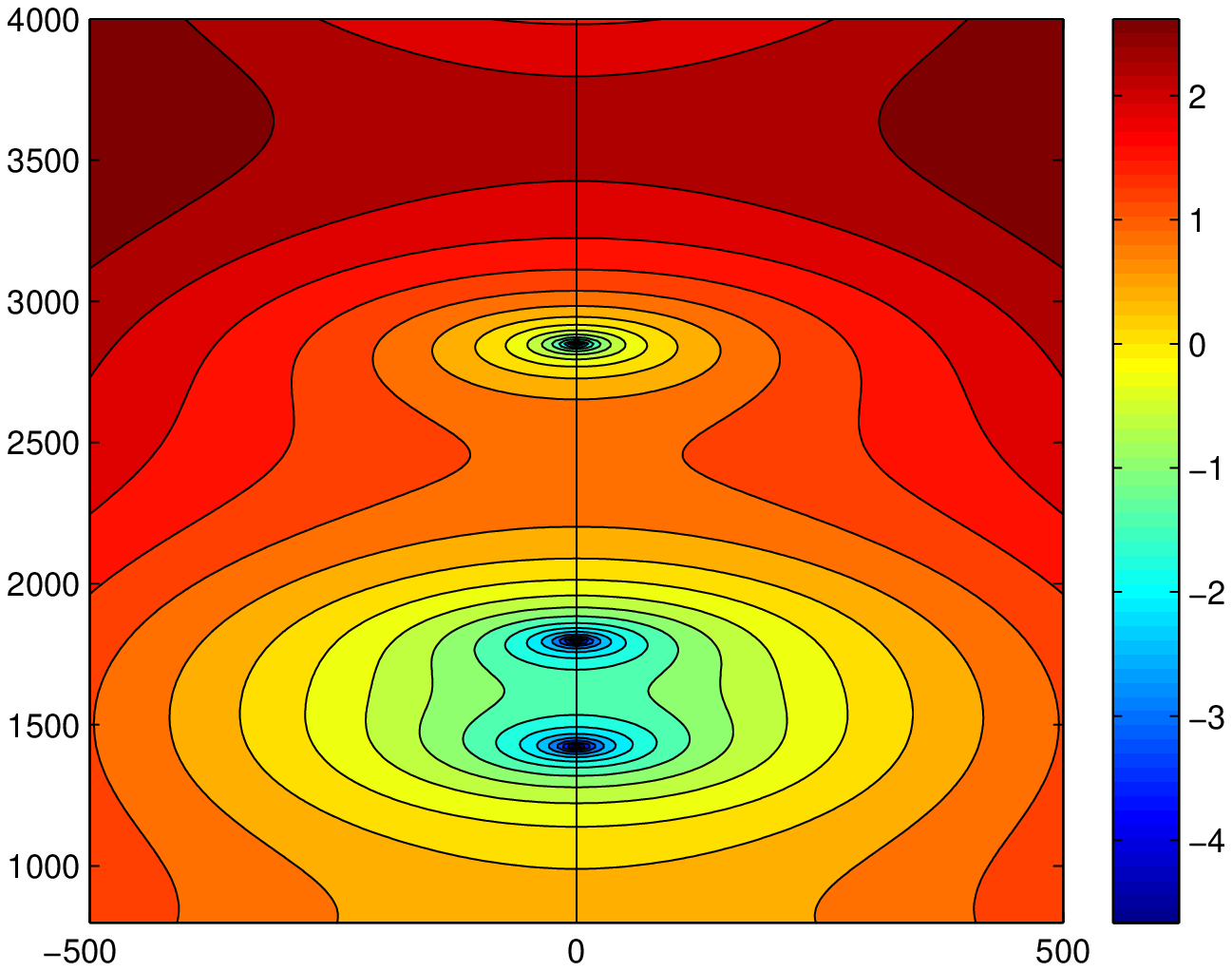}}
\centerline{$m=0$.}
\end{minipage}
\begin{minipage}{6.0cm}
\centerline{\includegraphics[height=5.5cm, angle=0]{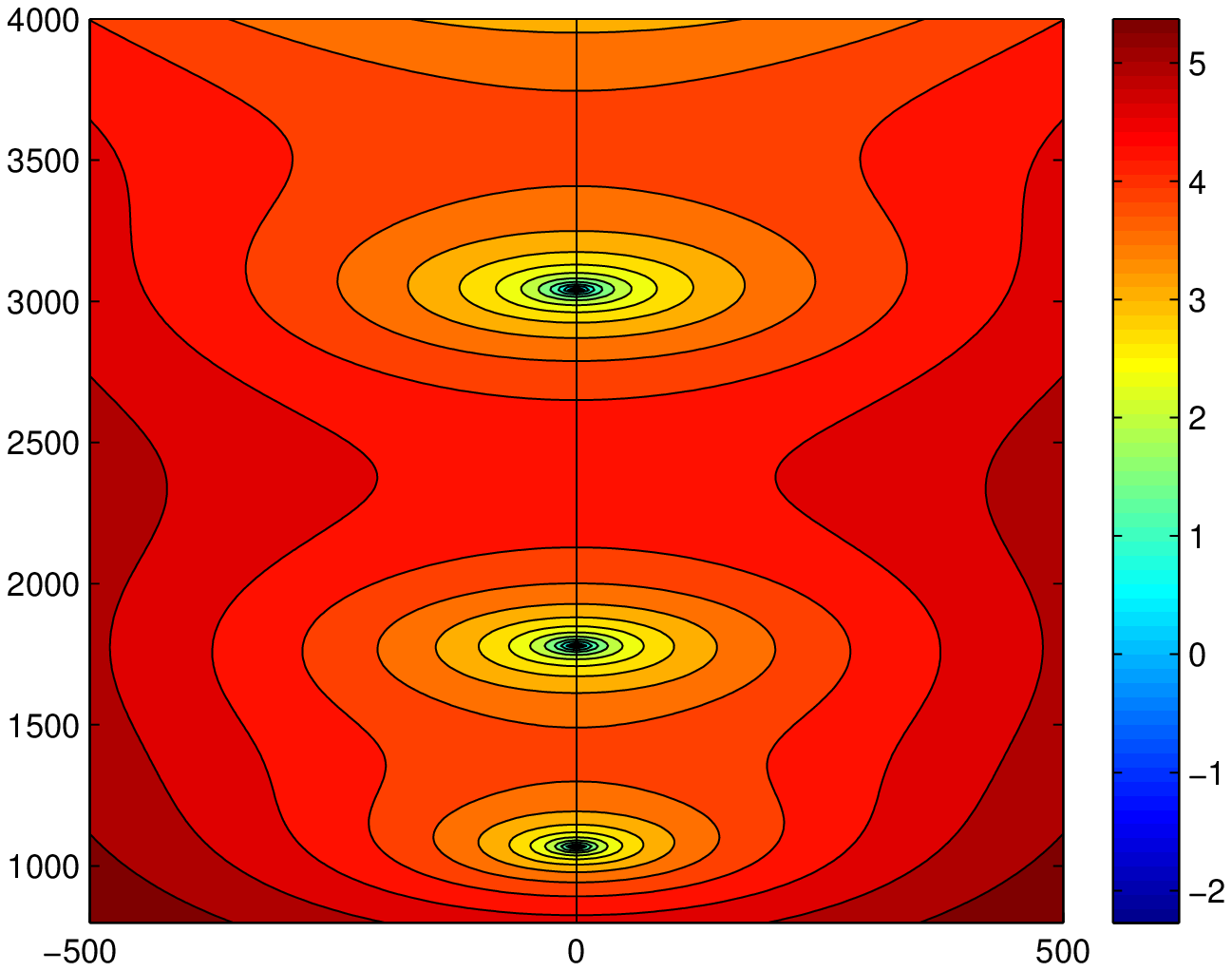}}
\centerline{$m=1$.}
\end{minipage}
\end{center}
\caption{Contour plots of $\log(|f_m(\lambda)|)$ for $m=0$ and $m=1$ with vanishing viscosity ($\nu=0$).}
\label{FIG:CONTOUR1}
\end{figure}
\hspace{-0.0cm}
\begin{figure}[H]
\begin{center}
\begin{minipage}{5.5cm}
\centerline{\includegraphics[height=5.5cm, angle=0]{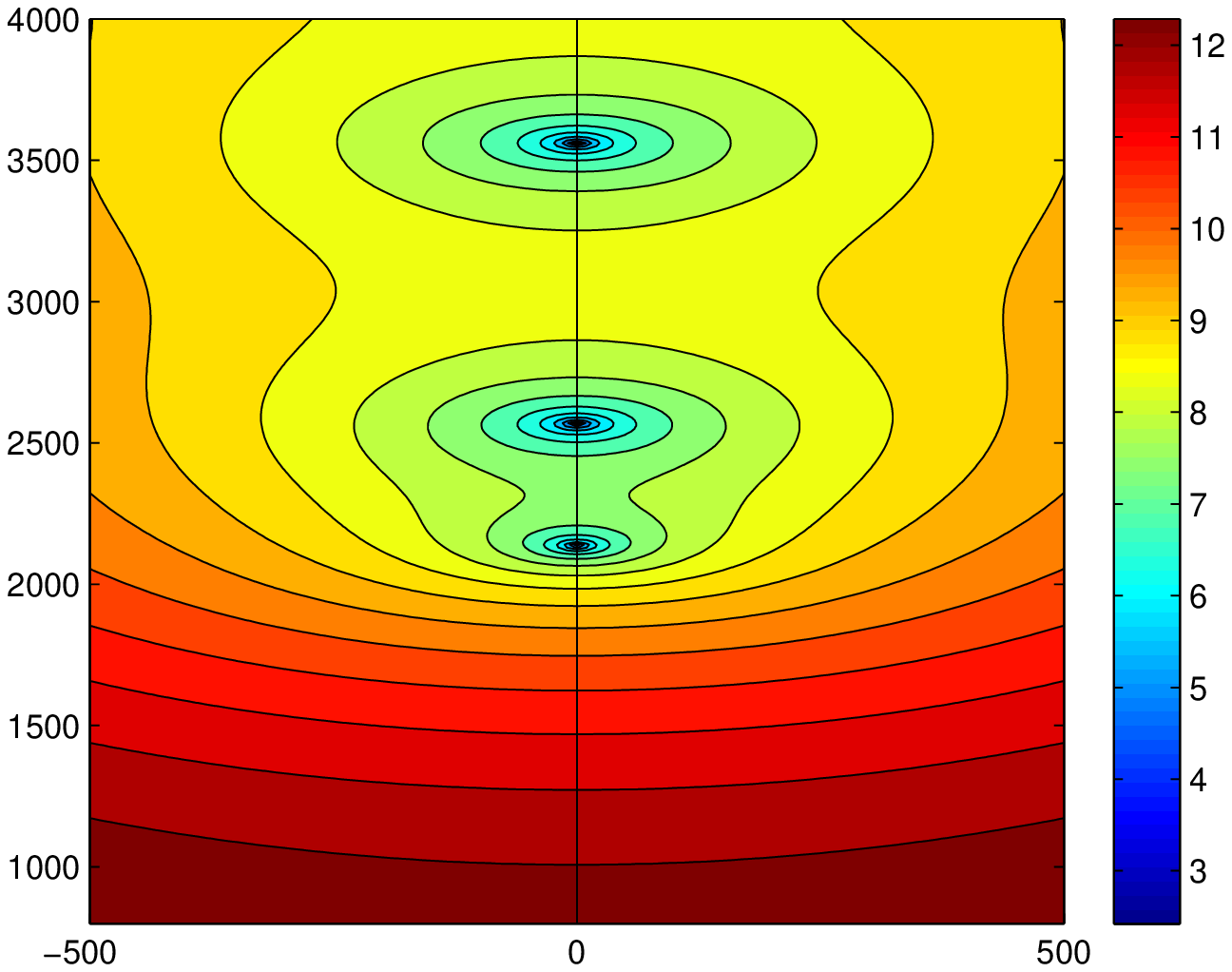}}
\centerline{$m=2$.}
\end{minipage}
\begin{minipage}{6.0cm}
\centerline{\includegraphics[height=5.5cm, angle=0]{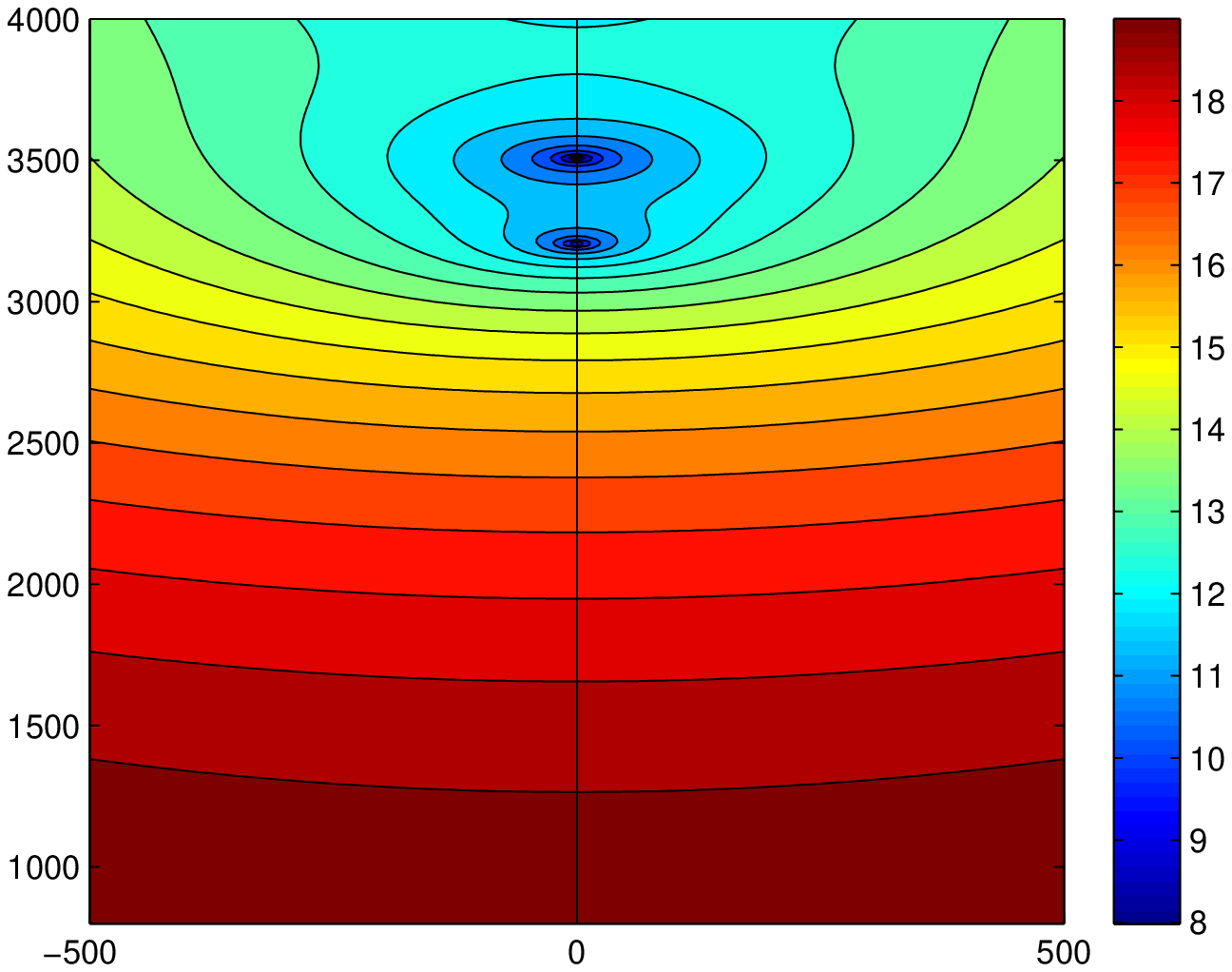}}
\centerline{$m=3$.}
\end{minipage}
\end{center}
\caption{Contour plots of $\log(|f_m(\lambda)|)$ for $m=2$ and $m=3$ with vanishing viscosity ($\nu=0$).}
\label{FIG:CONTOUR2}
\end{figure}
\hspace{-0.45cm}

Table \ref{TAB:TABLA1} shows the eigenvalues computed with the proposed
method on successively refined meshes that approximate those shown in
Figures~\ref{FIG:CONTOUR1} and \ref{FIG:CONTOUR2}. Accurate values of
the latter obtained with the MATLAB command $\tt{fminsearch}$ applied to
$|f_m(\lambda)|$ are also reported on the last line of the table as `exact'
eigenvalues. 

\begin{table}[H]
{\footnotesize
\begin{center}
\begin{tabular}{|c|c|c|c|c|c|c|}
\hline
 $m$     & $1$         & $0$        &  $1$       &$0$        &$2$         &$2$   \\ \hline
 $N=8$   & $1066.07\,i$&$1418.42\,i$&$1784.37\,i$&$1796.61\,i$&$2118.35\,i$&$2573.86\,i$ \\ 
 $N=16$  & $1067.78\,i$&$1422.52\,i$&$1781.49\,i$&$1797.09\,i$&$2131.94\,i$&$2569.90\,i$ \\
 $N=32$  & $1068.21\,i$&$1423.54\,i$&$1780.73\,i$&$1797.21\,i$&$2135.36\,i$&$2568.40\,i$ \\ 
 $N=64$  & $1068.33\,i$&$1423.79\,i$&$1780.55\,i$&$1797.23\,i$&$2136.22\,i$&$2568.09\,i$ \\ 
 Order & $2.00$      & $2.00$     &$1.99$      &$2.00$      &$1.99$      &$1.94$       \\         
 Exact & $1068.36\,i$&$1423.87\,i$&$1780.49\,i$ &$1797.24\,i$   &$2136.50\,i$&$2568.54\,i$ \\ \hline
\end{tabular}
\begin{tabular}{|c|c|c|c|c|c|}
\hline  
 $m$     & $0$         & $1$           &  $3$       &  $3$         &$2$           \\ \hline
 $N=8$   & $2807.28\,i$&$3021.26\,i$ &$3142.54\,i$ &$3492.47\,i$ &$3582.49\,i$ \\ 
 $N=16$  & $2837.76\,i$&$3037.38\,i$ &$3189.22\,i$ &$3503.56\,i$ &$3568.16\,i$ \\
 $N=32$  & $2845.75\,i$&$3041.02\,i$ &$3200.79\,i$ &$3506.19\,i$ &$3562.70\,i$ \\ 
 $N=64$  & $2848.88\,i$&$3041.89\,i$ &$3204.78\,i$ &$3506.83\,i$ &$3561.22\,i$ \\ 
 Order & $1.99$         & $2.02$      &$1.99$     &  $1.99$     &$1.94$           \\         
 Exact & $2849.56\,i$ &$3042.18\,i$&$3205.74\,i$&$3507.16\,i$   &$3560.72\,i$ \\ \hline
\end{tabular}
\caption{Computed and exact eigenvalues for dissipative fluids in a rigid cavity.}
\label{TAB:TABLA1}
\end{center}}
\end{table}

As predicted by the theory, these eigenvalues 
 are purely imaginary. The high accuracy of the
computed eigenvalues can be  observed from Table~\ref{TAB:TABLA1}
even for the coarsest mesh. We have used a least squares fitting
to estimate the convergence rate for each eigenvalue, which are also reported in Table~\ref{TAB:TABLA1}. 
A clear order $\mathcal{O}(h^2)$ can be seen in  all cases.

For the second test we have used the same physical parameters as above 
for both fluids, but considering now non vanishing viscosities.
In order to make the dissipation effects more visible, we have
used unrealistically large viscosity values:
$\nu_1=9\,\mbox{N/m$s^2$}$ and $\nu_2=1\,\mbox{N/m$s^2$}$.
We have repeated the scheme used  for the first test.
Figures \ref{FIG_CONTOUR3} and \ref{FIG_CONTOUR4} show the
localization of all the exact eigenvalues $\lambda$. Notice that
now  all $\lambda$ have  negative real parts (the decay rate) as predicted by the theory.

\begin{figure}[H]
\begin{center}
\begin{minipage}{5.5cm}
\centerline{\includegraphics[height=5.5cm, angle=0]{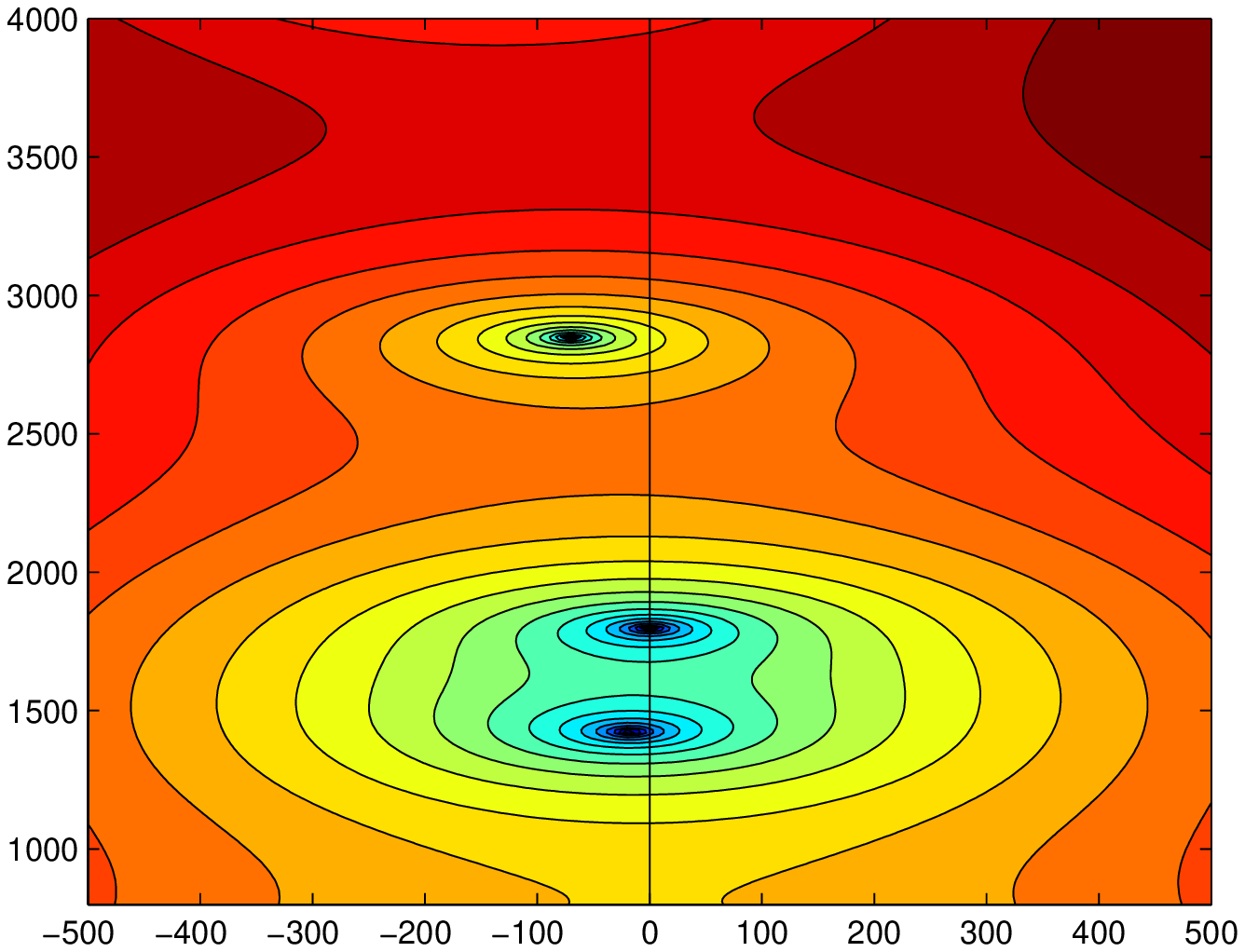}}
\centerline{$m=0$.}
\end{minipage}
\begin{minipage}{6.0cm}
\centerline{\includegraphics[height=5.5cm, angle=0]{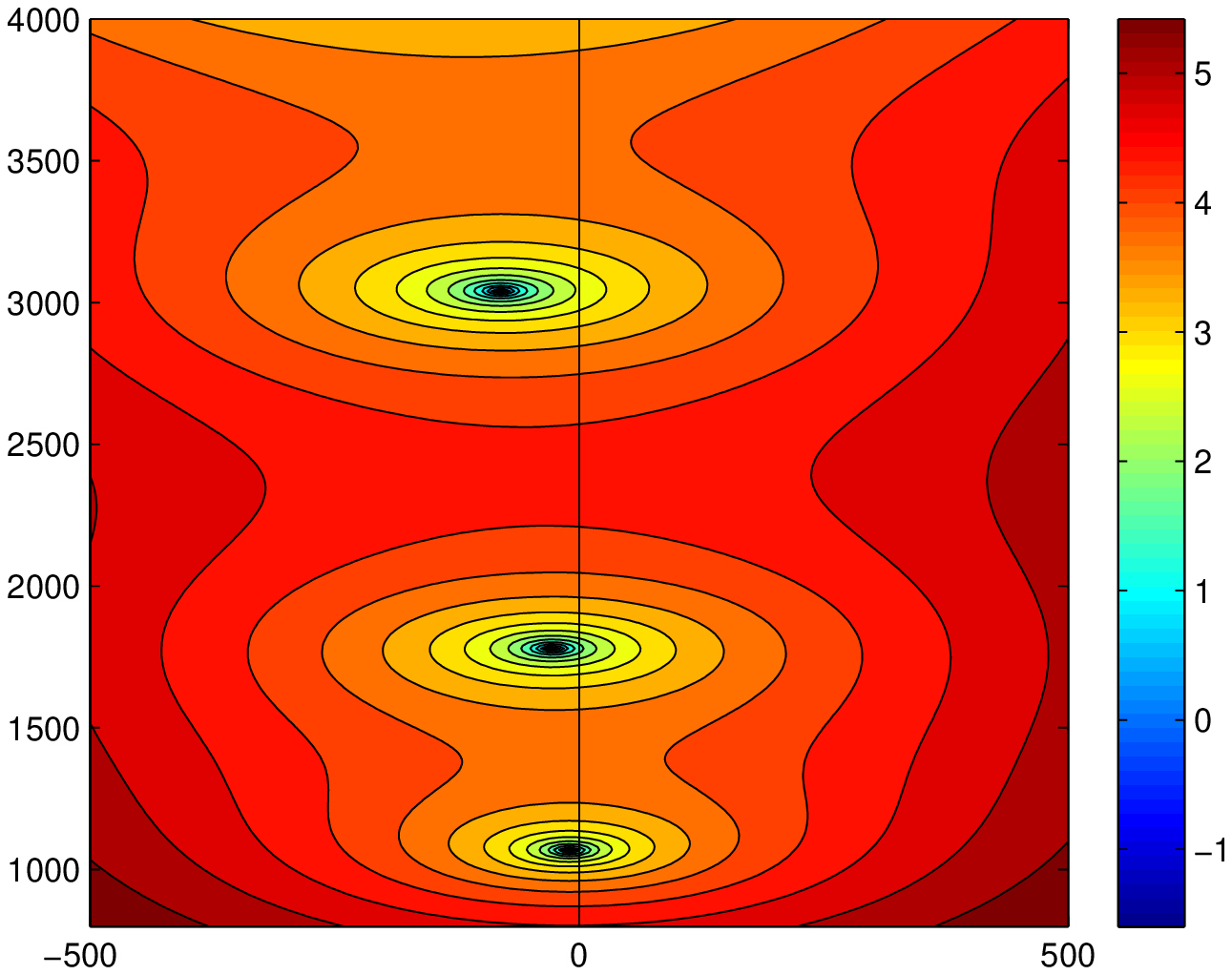}}
\centerline{$m=1$.}
\end{minipage}
\end{center}
\caption{Contour plots of $\log(|f_m(\lambda)|)$ for $m=0$ and $m=1$ with non-vanishing viscosity ($\nu\ne 0$).}
\label{FIG_CONTOUR3}
\end{figure}

\begin{figure}[H]
\begin{center}
\begin{minipage}{5.5cm}
\centerline{\includegraphics[height=5.5cm, angle=0]{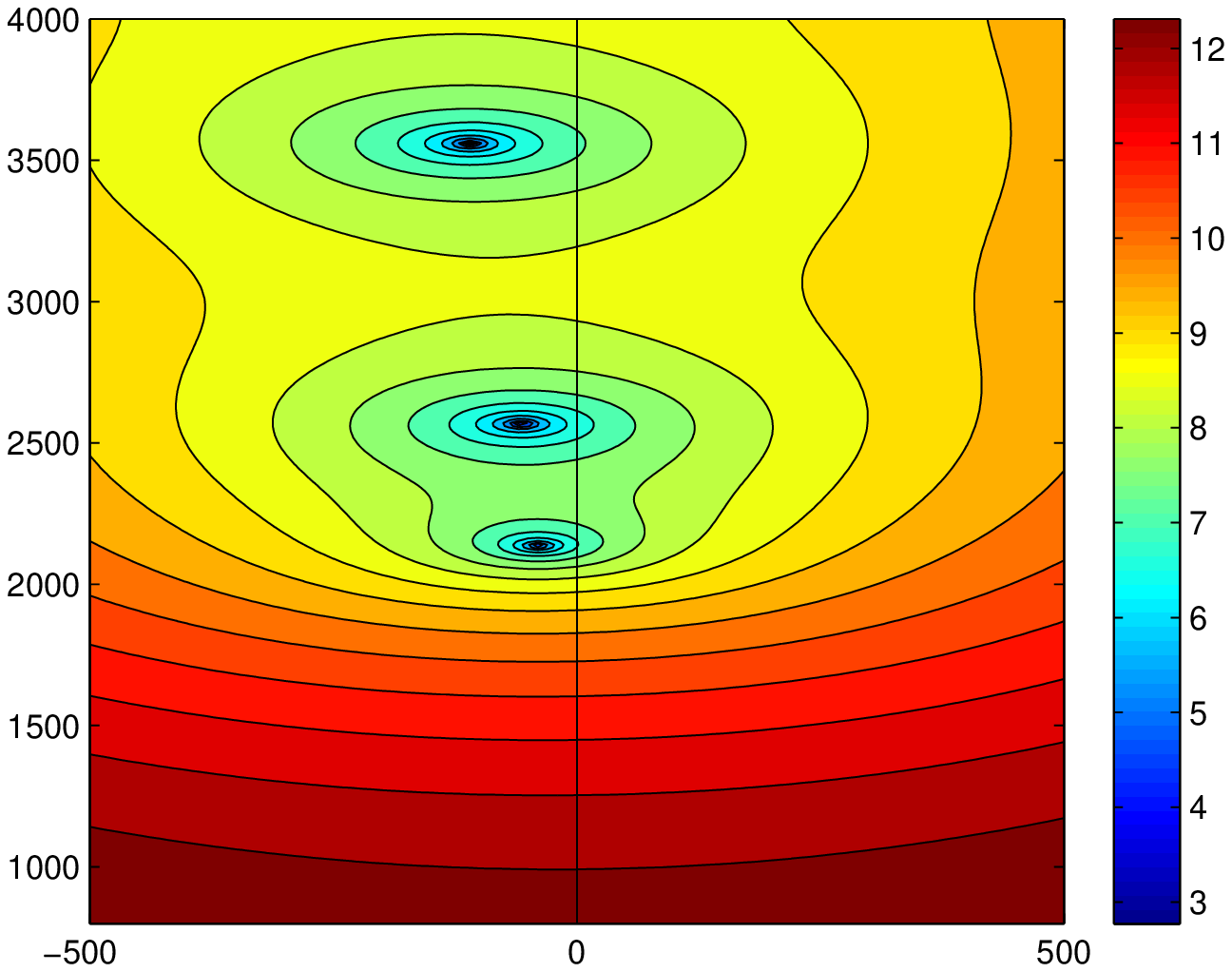}}
\centerline{$m=2$.}
\end{minipage}
\begin{minipage}{6.0cm}
\centerline{\includegraphics[height=5.5cm, angle=0]{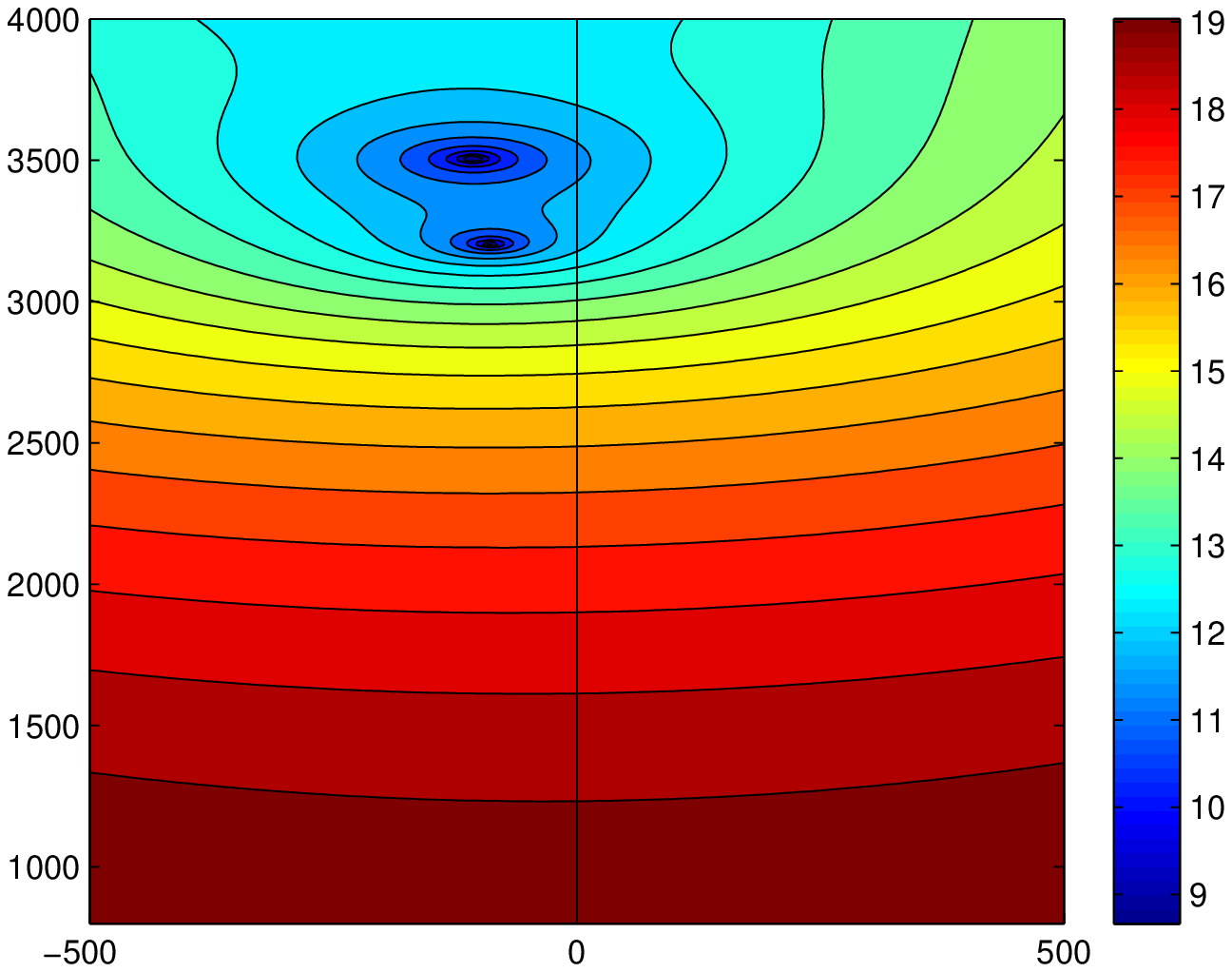}}
\centerline{$m=3$.}
\end{minipage}
\end{center}
\caption{Contour plots of $\log(|f_m(\lambda)|)$ for $m=2$ and $m=3$ with non-vanishing viscosity ($\nu\ne 0$).}
\label{FIG_CONTOUR4}
\end{figure} 
\vspace{-0.5cm}
We report in Table~\ref{TABLA2} the computed and `exact'
eigenvalues and the estimated convergence rates,
which are in accordance with the theory once again.
\vspace{-0.5cm}

\begin{table}[H]
{\footnotesize
\begin{center}
\begin{tabular}{|c|c|c|c|c|}
\hline
 $m$  &$1$                    & $0$                    & $1$                  &$0$    \\ \hline
 $N=8$ &$-9.83 + 1066.03\,i$  &$-17.39 + 1418.31\,i$   &$-27.54 + 1784.16\,i$ &$-0.05 + 1796.61\,i$       \\ 
 $N=16$&$-9.86 + 1067.74\,i$  &$-17.49 + 1422.41\,i$   &$-27.45 + 1781.27\,i$ &$-0.05 + 1797.08\,i$       \\ 
 $N=32$&$-9.87 + 1068.17\,i$  &$-17.51 + 1423.43\,i$   &$-27.43 + 1780.53\,i$ &$-0.05 + 1797.20\,i$       \\ 
 $N=64$&$-9.87 + 1068.38\,i$  &$-17.52 + 1423.78\,i$   &$-27.42 + 1780.34\,i$ &$-0.05 + 1797.23\,i$         \\ 
 Order&     $2.00$          &      $2.00$            &      $1.99$          &     $2.00$              \\ 
 Exact&$-9.87 + 1068.32\,i$ &$-17.52 + 1423.76\,i$   &$-27.42 + 1780.27\,i$ &$-0.05 + 1797.24\,i$       \\ \hline         
\end{tabular}
\begin{tabular}{|c|c|c|c|c|}
\hline
 $m$ &$2$ & $2$ & $0$ &$1$    \\ \hline
 $N=8$&$-38.82 + 2118.00\,i$   &$-57.31 + 2573.22\,i$   &$-68.17 + 2806.45\,i$ &$-78.96 + 3020.24\,i$     \\ 
 $N=16$&$-39.32 + 2131.58\,i$  &$-57.13 + 2569.26\,i$   &$-69.65 + 2836.91\,i$ &$-79.80 + 3036.34\,i$       \\ 
 $N=32$&$-39.44 + 2135.00\,i$ &$-57.06 + 2567.76\,i$   &$-70.05 + 2844.09\,i$  &$-80.00 + 3039.96\,i$      \\ 
 $N=64$&$-39.58 + 2135.95\,i$  &$-57.04 + 2567.36\,i$   &$-70.15 + 2846.92\,i$ &$-80.04 + 3040.84\,i$       \\ 
 Order &$1.99$                 &$1.94$                  &$1.99$                &       $2.02$       \\ 
 Exact&$-39.49 + 2136.14\,i$  &$-57.04 + 2567.22\,i$   &$-70.18 + 2847.60\,i$ &$-80.06 + 3041.13\,i$      \\ \hline        
\end{tabular}
\begin{tabular}{|c|c|c|c|}
\hline
 $m$     &$3$                     & $3$ & $2$      \\ \hline
 $N=8$   &$-85.43 + 3141.39\,i$   &$-105.51 + 3490.88\,i$    & $-111.02 + 3580.77\,i$       \\ 
 $N=16$  &$-87.99 + 3188.01\,i$   &$-106.18 + 3501.95\,i$    & $-110.14 + 3566.47\,i$      \\ 
 $N=32$  &$-88.62 + 3199.56\,i$   &$-106.34 + 3504.57\,i$    & $-109.80 + 3561.01\,i$     \\ 
 $N=64$  &$-88.88 +3202.45\,i$    &$-106.38 + 3505.22\,i$    & $-109.70 + 3559.53\,i$         \\ 
 Order &   $1.99$               & $1.99$                  &$1.94$       \\ 
 Exact &$-88.84 + 3203.41\,i$   &$-106.40 + 3505.44\,i$    &$ -109.68 + 3559.03\,i$     \\ \hline       
\end{tabular}
\caption{Computed and exact eigenvalues for dissipative fluids in a rigid cavity.}
\label{TABLA2}
\end{center}}
\end{table}

\end{document}